\documentclass{elsarticle}
\usepackage{pict2e}
\usepackage{amsmath}
\usepackage{amsthm}
\usepackage{amssymb}
\usepackage{xspace}

\newcommand{\N}{\mathbb{N}}
\newcommand{\Z}{\mathbb{Z}}
\newcommand{\cA}{{\mathcal{A}}}
\newcommand{\cG}{\mathcal{G}}
\newcommand{\cH}{\mathcal{H}}
\newcommand{\cS}{\mathcal{S}}
\newcommand{\cT}{\mathcal{T}}
\newcommand{\cL}{\mathcal{L}}
\newcommand{\cU}{\mathcal{U}}
\newcommand{\cY}{\mathcal{Y}}
\newcommand{\cZ}{\mathcal{Z}}
\newcommand{\emptyword}{\epsilon}
\newcommand{\ee}{\mathsf e}
\newcommand{\slp}{{\sc slp}\xspace}
\newcommand{\slps}{{\sc slp}s\xspace}
\newcommand{\tslp}{{\sc tslp}\xspace}
\newcommand{\tslps}{{\sc tslp}s\xspace}
\newcommand{\cslp}{{\sc cslp}\xspace}
\newcommand{\cslps}{{\sc cslp}s\xspace}
\newcommand{\tcslp}{{\sc tcslp}\xspace}
\newcommand{\tcslps}{{\sc tcslp}s\xspace}
\newcommand{\slex}{\mathsf{slex}}
\newcommand{\nf}{\mathsf{nf}}
\newcommand{\lex}{\mathsf{lex}}
\newcommand{\sgn}{\mathrm{sgn}}

\newcommand{\wnf}{w_{\mathsf{nf}}}

\newcommand{\vnf}{v_{\mathsf{nf}}}
\newcommand{\znf}{z_{\mathsf{nf}}}
\newcommand{\JJ}{J}
\newcommand{\CC}{\mathsf{C}}
\newcommand{\DD}{\mathsf{D}}
\newcommand{\jj}{\mathsf{j}}
\newcommand{\Area}{\mathsf{Area}}
\newcommand{\rel}{\mathsf{rel}}
\theoremstyle{plain}
\newtheorem{theorem}{Theorem}[section]
\newtheorem{theoremA}{Theorem}

\newtheorem{proposition}[theorem]{Proposition}
\newtheorem{lemma}[theorem]{Lemma}
\newtheorem{corollary}[theorem]{Corollary}
\newtheorem{remark}[theorem]{Remark}
\theoremstyle{definition}

\newtheorem{definition}[theorem]{Definition}
\newenvironment{proofof}[1]{\normalsize {\it Proof of #1}.}{{\hfill $\Box$}}

\newcommand{\ff}{\mathsf{f}} 
\newcommand{\val}{\mathsf{val}}
\newcommand{\Ht}{\mathsf{h}} %
\newcommand{\tHt}{\mathsf{t}}
\newcommand{\tD}{\mathsf{d}}
\newcommand{\td}{\mathsf{d}} % now function and constant are not distinct
\newenvironment{mylist}{\begin{list}{}{
\setlength{\parskip}{0mm}
\setlength{\topsep}{2mm}
\setlength{\parsep}{0mm}
\setlength{\itemsep}{0.5mm}
\setlength{\labelwidth}{7mm}
\setlength{\labelsep}{3mm}
\setlength{\itemindent}{0mm}
\setlength{\leftmargin}{12mm}
\setlength{\listparindent}{6mm}
}}{\end{list}}

\begin{document}
\begin{frontmatter}
\title{The compressed word problem in relatively hyperbolic groups}
\author{Derek Holt and Sarah Rees\\ \vspace{1mm}
\textit{Dedicated to Patrick Dehornoy, our collaborator and friend,
from whom we learned a lot}}

\date{13th July 2021, Warwick}

\begin{abstract}
We prove that the compressed word problem in a group that is hyperbolic
relative to a collection of free abelian subgroups is solvable in
polynomial time.
\end{abstract}

\end{frontmatter}

\section{Introduction}

The main result of \cite{HLS} is that the compressed word problem in a
hyperbolic group is solvable in polynomial time. Here we generalise this
result to a group hyperbolic relative to a set of free abelian subgroups.

\begin{theoremA}\label{thm:main}
The compressed word problem for a group that is hyperbolic relative to
a collection of free abelian subgroups is solvable in polynomial time.
\end{theoremA}

We prove the theorem by extending the arguments of \cite{HLS} from hyperbolic
to relatively
hyperbolic groups. Our principal source for results about this class of
groups is \cite{AC}. In particular, we use the automaticity
of groups of that type, proved in \cite[Theorem 7.7]{AC}; however we need to 
construct a new asynchronously automatic structure, with particular 
properties that we need.

We believe that our result remains true if we substitute arbitrary (finitely
generated) abelian groups for free abelian groups in the theorem statement
but (as we shall explain after the definition of the \emph{components}
of a word in Section~\ref{sec:relhyp}) extending the proof in that way would
result in further technical difficulties in a proof that is already highly
technical, so we decided not to attempt it here.
It might be more interesting to try to extend
the result to groups hyperbolic with respect to a collection of virtually
abelian subgroups, but we are not currently able to extend our methods to
cover that case in general.

We introduce basic concepts and notation, including the definitions of straight line programs (\slp) and the compressed word problem, in 
Section~\ref{sec:defs_notation}, which follows this
introduction. Section~\ref{sec:relhyp} contains the definition of relatively
hyperbolic groups (following Osin \cite{Osin}) and properties of those
that we shall need in the article. Section~\ref{sec:slp_background} provides
basic background material on \slps, 
Section~\ref{sec:slps_ab} gives some results for \slps associated with 
finitely generated abelian groups, 
Section~\ref{sec:gammahat_geom} examines the geometry of the compressed 
Cayley graph $\widehat{\Gamma}$ of a relatively hyperbolic group $G$,
and Section~\ref{sec:slps_relhyp_props} establishes
some results for \slps associated with such a group.
The main result of the article, Theorem~\ref{thm:main}, is proved across the final
two sections, Sections~\ref{sec:convert} and \ref{sec:slextcslp}, as the
concatenations of two theorems, Theorem~\ref{thm:slextcslp} and Theorem~\ref{thm:convert}.

The authors would like to thank Saul Schleimer, for introducing us to the study
of the compressed word problem, especially for hyperbolic groups \cite{HLS},
and Yago Antolin, for some very helpful discussions on the properties of
relatively hyperbolic groups and their automatic structures.
We are also grateful to an anonymous referee for a careful reading of
the paper and several helpful comments and suggestions.

\section{Definitions and notation}
\label{sec:defs_notation}
To a large extent (but not entirely) our notation and definitions follow
\cite{HLS}.

\subsection{Words}
For a finite set $\Sigma$ (which we call an {\em alphabet}), we define a
{\em word} $w$ over $\Sigma$ to be a string $x_0\cdots x_{n-1}$, where each
$x_i$ is in $\Sigma$. We denote by $|w|$ the length $n$ of $w$,
by $\emptyword$ the empty word (of length 0), and
for $0 \leq i < j \leq n$
we denote by $w[i:j)$ the substring $x_i\cdots x_{j-1}$,
which we call a {\em subword} of $w$ (in \cite{HLS}, and elsewhere,
such a substring is called a {\em factor}).
We abbreviate the prefix $w[0:j)$ of $w$ as $w[:j)$, its suffix $w[i:n)$ as $w[i:)$,
and $w[i:i+1)=x_i$ as $w[i]$, and also consider $\emptyword$
to be a subword of $w$.

For words $v,w \in \Sigma^*$,
we write $v=w$ if $v$ and $w$ are equal as strings
and, when $\Sigma$ is a generating set of a group $G$, we write $v=_G w$
if $v$ and $w$ represent the same element of $G$.
All group generating sets $\Sigma$ in this article will be assumed to be
inverse closed (that is, $x \in \Sigma \Rightarrow x^{-1} \in \Sigma$).

Suppose that $\Sigma$ is an ordered finite generating set for a group $G$
and that $w \in \Sigma^*$ represents an element $g \in G$. Then we define
$\slex(w)$ (or $\slex_\Sigma(w))$ to be the shortlex minimal word representing
$g$; that is, $\slex(w)$ is the lexicographically least among the shortest
words that represent the same group element as $w$.

\subsection{Straight-line programs}
Let $\Sigma$ be a finite alphabet and $V$ a finite set with
$V \cap \Sigma = \emptyset$.
Let  $\rho: V \to (V \cup \Sigma)^*$ be a map and extend the definition of
$\rho$ to $(V \cup \Sigma)^*$ by defining
$\rho(a)=a$ for all $a \in \Sigma \cup \{\emptyword\}$
and $\rho(uv) = \rho(u)\rho(v)$ for all $u,v \in  (V \cup \Sigma)^*$.
We define the associated  binary relation $\succeq$ on $V$ by
$A \succeq B$ whenever the symbol $B$ occurs within the string $\rho^k(A)$,
for some $k \geq 0$.

We define a {\em straight-line program} (\slp for short) over an alphabet 
$\Sigma$ to be a triple $\cG = (V,S,\rho)$, with $S \in V$ and
$\rho: V \to (V \cup \Sigma)^*$ a map such that the associated binary
relation $\succeq$ on $V$ is acyclic,
that is the corresponding directed graph contains no directed cycles.
The set $V$ is called the set of {\em variables} of $\cG$, and $S$ is
called the {\em start variable}.
Where necessary, we write $V_\cG$,  $S_\cG$, $\rho_\cG$, rather than simply
$V,S,\rho$.

An \slp $\cG$ is naturally associated with a
context-free grammar $(V,\Sigma, S, P)$, where $P$ is the
set of all productions $A \to \rho(A)$ with $A \in V$,
and we will often use the name $\cG$ also for this grammar.
It follows from the definition of an \slp that this associated grammar 
derives exactly one terminal word, which we call the {\em value} of $\cG$
and denote by $\val(\cG)$.

\slps are used to provide succinct representations of words that contain many
repeated substrings. For instance, the word $(ab)^{2^n}$ is the value of
the \slp $\cG = (\{A_0, \ldots,A_n\},\rho,A_0)$ with
$\rho(A_n) = ab$ and $\rho(A_{i-1}) = A_i A_i$ for $0 < i \leq n$.
We provide more background on \slps in Section~\ref{sec:slp_background}.

%%%%%%Definition of the compressed word problem%%%%%
\subsection{The compressed word problem}
The {\em compressed word problem} for a finitely generated group $G$ with
the finite symmetric generating set $\Sigma$
is the following decision problem:
\begin{description}
\item[Input:]  an \slp $\cG$ over the alphabet $\Sigma$.
\item [Question:] does $\val(\cG)$ represent the group identity of $G$?
\end{description}
It is an easy observation that the computational complexity of the compressed
word problem for $G$ is independent of the choice of generating set $\Sigma$;
more precisely, if $\Sigma'$ is another finite symmetric generating set for $G$,
then the compressed word problem for $G$ with respect to $\Sigma$ is log-space
reducible to the compressed word problem for $G$ with respect to $\Sigma'$
\cite[Lemma~4.2]{Loh14}.
So, when proving that the compressed word problem for $G$ is solvable in
polynomial time, we are free to choose 
whichever finite symmetric generating
set of $G$ is most convenient for the purpose.

%%%%%%%Fellow travelling and automatic groups
\subsection{Fellow travelling and automatic groups}
Suppose that $\Gamma=\Gamma(G,\Sigma)$ is the Cayley graph for a group $G$
with finite symmetric generating set $\Sigma$. Let $v,w$ be words over $\Sigma$,
and let $\gamma_v,\gamma_w$ be the
paths traced out in $\Gamma$ by $v,w$ from the identity vertex $1$ of $\Gamma$.
For $0 \le i < |v|$, 
we denote the vertex of $\Gamma$ labelled $v[i]$ by
$\gamma_v[i]$, and similarly for $w[i]$ and $\gamma_w[i]$. (So, in particular,
$\gamma_v[0]=\gamma_w[0]$ is the identity vertex.)
For the following definition, we define
$\gamma_v[i] := \gamma_v[|v|-1]]$ for integers $i \ge |v|$.

We say that the words $v$ and $w$ {\em fellow travel} at distance $k$
(or, more briefly, {\em $k$-fellow travel}) if
$d_\Gamma(\gamma_v[i],\gamma_w[i]) \le k$ for all $i \ge 0$.
In other words, the distance in $\Gamma$ between the vertices on $\gamma_v$ and
$\gamma_w$ at the ends of subpaths traced out by the subwords $v[:i)$ and
$w[:i)$ of $v$ and $w$
%at distance $i$ along that path from the identity vertex
is at most $k$.
In this situation we also say that the paths $\gamma_v$ and $\gamma_w$
$k$-fellow travel.

We can extend this terminology to  paths $\gamma_v,\gamma_w$ that do not start
at the same vertex in $\Gamma$. In particular, if two such paths $k$-fellow
travel, then their start points and also their end points are
at distance at most $k$ from each other.

We say that the words $v$ and $w$ {\em asynchronously fellow travel}
at distance $k$ if we can choose sequences $(i_0=0,i_1,\ldots,i_n=|v|-1)$ and
$(j_0=0,j_1,\ldots,j_n=|w|-1)$, with $i_{t+1}\in \{i_t,i_t+1\}$,
$j_{t+1} \in \{j_t,j_t+1\}$ for each $t$, such that
$d_\Gamma(\gamma_v[i_t],\gamma_w[j_t])\leq k$ for each $t$ with $0 \le t \le n$,
and again we may also apply this
terminology to the paths that 
they trace out in $\Gamma$.

We call the pairs of vertices $\gamma_v[i_t],\gamma_w[j_t]$
\emph{corresponding vertices} in the fellow travelling.
Note that a vertex on one of the paths can have more than one corresponding
vertex on the other path, but in that case the corresponding vertices
are the vertices lying on a contiguous subpath of the other path.

We say that $G$ is {\em automatic} if there exists a finite state automaton $A$ over
$\Sigma$
for which every element of $G$ has at least one representative word in $L(A)$,
as well as an integer $k$ such that,
whenever $v,w \in L(A)$ and either $v=_G w$ or $v=_G wx$ with $x \in \Sigma$,
then $v,w$ fellow travel at distance $k$. We say that $G$ is
{\em asynchronously automatic} if the same is true but with an asynchronous fellow traveller property, and {\em (asynchronously) biautomatic}
if, in addition, for words $u,w$ that satisfy $xu =_G w$ with $x \in \Sigma$,
that paths traced out by $u,w$ from the vertices $x$ and $1$, respectively,
(asynchronously) fellow travel at distance $k$.

We call the pair $(A,k)$ an {\em automatic structure} (or {\em asynchronously automatic structure}) for $G$. An automatic structure $(A,k)$ is called 
{\em geodesic} or {\em shortlex} if each word in $L(A)$ is a representative 
of minimal length, or minimal within the shortlex word ordering,
of the element that it represents, respectively.
We say that $(A,k)$ is a {\em structure with uniqueness}
if it contains a unique representative of each element of $G$.
We refer to \cite{ECHLPT} for basic properties of automatic structures and
automatic groups.

%%%%%%Relatively hyperbolic groups
\section{Relatively hyperbolic groups}\label{sec:relhyp}
The purpose of this section is to give the definition of a relatively
hyperbolic group and to list the properties that we shall need in this
article.
The properties we need are proved in the article \cite{AC}, and build on
results of \cite{Osin}. 
We have used Osin's definition of relatively hyperbolicity;
it is proved in \cite[Theorem 1.5]{Osin} that (for finitely generated groups,
as in our case) this is equivalent to the definition of \cite{Bowditch},
also to the definition of \cite{Farb} combined with the
Coset Penetration Property (see below), called strong relative hyperbolicity
in \cite{Farb}.
Below we have (essentially) used notation and statements from \cite{AC}.

We suppose that $\Sigma$ is a finite generating set for a group $G$, and that
$\{H_i : i \in \Omega\}$ is a finite collection of subgroups of a
$G$, which we call the collection of
\emph{parabolic subgroups} of $G$.
Define $\cH := \bigcup_{i \in \Omega} (H_i \setminus \{1\})$, and 
$\widehat{\Sigma} := \Sigma \cup \cH$.
We let $\Gamma:=\Gamma(G,\Sigma)$ and
$\widehat{\Gamma}=\widehat{\Gamma}(G,\widehat{\Sigma})$ be the Cayley graphs
for $G$ over $\Sigma$ and $\widehat{\Sigma}$, respectively.
(So $\widehat{\Gamma}$ has the same vertices as $\Gamma$ but more edges
than $\Gamma$.)
We call a word over $\Sigma$ (or $\widehat{\Sigma}$) \emph{geodesic}
if it labels a geodesic path in
$\Gamma$ (or $\widehat{\Gamma}$).

Following \cite[Definition 2.5]{AC} and \cite[Section 1.2]{Osin},
we define $F$ to be the free product of groups
\[ F:= (\ast_{ i \in \Omega} H_i) \ast F(\Sigma) \] 
and suppose that a finite subset $R$ of $F$ exists whose normal closure in $F$
is the kernel of the natural map from $F$ to $G$; in that case we say that
$G$ has the {\em finite presentation} 
\[ \left\langle X \cup \left. \bigcup_{i \in \Omega} H_i\, \right\vert  R \right\rangle \] 
{\em relative to} the collection of subgroups  $\{ H_i : i \in \Omega \}$.
Now if $u$ is a word over $\widehat{\Sigma}$ that represents the
identity in $G$, then $u$ is equal within $F$ to a product of the form
\[  \prod_{j=1}^n f_j r_j^{\eta_j} f_j^{-1}, \]
with $r_j \in R, f_j \in F$ and $\eta_j = \pm 1 $ for each $j$.
The smallest possible value of $n$ in any such expression of this type for $u$
is called the {\em relative area} of $u$, denoted by $\Area_\rel(u) $.

We say that $G$ is \emph{hyperbolic relative to}
the collection of subgroups $\{H_i\}$ if it has a finite relative
presentation as above and a constant $C \geq 0$ such that
\[ \Area_\rel(u) \leq C |u| \]
for all words  $u$ over $\widehat{\Sigma}$ that represent the identity in $G$.

We note that if $G$ is relatively hyperbolic then the graph $\widehat{\Gamma}$
is Gromov-hyperbolic. Note also that, by \cite[Proposition 2.36]{Osin},
the intersection $H_i \cap H_j$ for $i \ne j$ is finite.

The notation and results that follow are all taken from \cite{AC}.
Given a path $p$ in $\widehat{\Gamma}$,  we say that the path $p$
\emph{penetrates} the left coset $gH_i$ if $p$ contains an edge
labelled by an element of $H_i$ that connects two vertices of $gH_i$.
An $H_i$--\emph{component} of such a path is defined to be 
a non-empty maximal subpath of $p$ that is labelled by a word in $H_i^*$.
Two components $s$ and $r$ (not necessarily of the same path) are
{\em connected} if both are $H_i$-components for some $H_i$, and if the start
points of both paths lie in the same left coset $gH_i$ of $H_i$.

A path $p$ is said to \emph{backtrack} if 
$p=p'srs'p''$ where $s,s'$ are $H_i$--components, and the word labelling $r$ 
represents an element of $H_i$; if no such decomposition of $p$ exists, then
$p$ is \emph{without backtracking}.
A path $p$ is said to \emph{vertex backtrack} if it contains a subpath of
length greater than 1 labelled by a word that represents an element of some
$H_i$; otherwise $p$ is said to be \emph{without vertex backtracking}.
We note that if a path does not vertex backtrack then it does not backtrack
and all of its components have length 1.

We denote the start and end points of a path  $p$ in $\widehat{\Gamma}$ by
$p_-$ and $p_+$, respectively, and
say that paths $p$, $q$ in $\widehat{\Gamma}$ are \emph{$k$-similar} if
$\max\{d_\Gamma(p_-,q_-),d_\Gamma(p_+,q_+)\} \le k$.
The following fundamental result about $k$-similar paths in $\widehat{\Gamma}$,
proved as \cite[Theorem 3.23]{Osin}, is also stated as \cite[Theorem 2.8]{AC}. 

\begin{proposition}\label{prop:bcpp} \cite[Theorem 3.23]{Osin}.
(Bounded Coset Penetration Property)
Let $G$ be relatively hyperbolic, as above.
For any $\lambda \ge 1$, $c \ge 0$, $k \ge 0$, there exists a constant
$\ee = \ee(\lambda,c,k)$ such that, for any two $k$-similar paths
$p$ and $q$ in $\widehat{\Gamma}$ that are $(\lambda,c)$--quasigeodesics and
do not backtrack, the following conditions hold.
\begin{mylist}
\item[(1)] The sets of vertices of $p$ and $q$ are contained in the
closed $\ee$-neighbourhoods of each other in $\Gamma$.
\item[(2)] Suppose that, for some $i$, $s$ is an $H_i$-component of $p$ with
$d_\Gamma(s_-,s_+) > \ee$; then there exists an $H_i$-component
of $q$ that is connected to $s$.
\item[(3)] Suppose that $s$ and $t$ are connected $H_i$-components of $p$
and $q$, respectively. Then $s$ and $t$ are $\ee$-similar.
\end{mylist}
\end{proposition}

The next three results are derived in \cite{AC} from the Bounded Coset
Penetration Property. 

We define the \emph{components} of a word $w \in \Sigma^*$ to be the
nonempty subwords of $w$ of maximal length that lie in $(\Sigma \cap H_i)^*$
for some parabolic subgroup $H_i$.
In general, since $H_i \cap H_j$ is finite for $i \ne j$, it is possible
for the end of one component to overlap the beginning of the next, where
the overlapping generators lie in a finite intersection. 
In this paper, we shall generally be assuming that $H_i \cap H_j$ is
trivial for $i \ne j$, in which case the components are necessarily disjoint.
This holds in particular when the $H_i$ are free abelian groups, and it is
the main reason why we have not attempted to generalise our main theorem to
groups that are hyperbolic relative to arbitrary finitely generated abelian
groups. We strongly believe that such a generalisation would be possible,
but it might involve significant additional technicalities, of a similar nature
to those involved in \cite{AC}.

Let $w := \alpha_0u_1\alpha_1u_2 \cdots u_n \alpha_n$, where the $u_j$
are its components.  Then, following \cite[Construction 4.1]{AC}, we define
the {\em derived word}
$\hat{w} := \alpha_0h_1\alpha_1h_2 \cdots h_n \alpha_n \in
\widehat{\Sigma}^*$,
where each $h_j$ is the element of a parabolic subgroup represented by $u_j$. 
So the components of paths in $\Gamma$ and $\widehat{\Gamma}$ labelled by $w$
and $\hat{w}$ are labelled by the subwords $u_i$ and $h_i$ of $w$ and
$\hat{w}$, respectively.

A word over $\Sigma$ is said to have a \emph{parabolic shortening}
if, for some $i$, it has a component over $\Sigma \cap H_i$ that is
non-geodesic; otherwise it has {\em no parabolic shortenings}.

\begin{proposition} \label{prop:genset} \cite[Lemma 5.3, Theorems 7.6, 7.7]{AC}
Let $G$ be a finitely generated group, hyperbolic with respect to a
family of subgroups $\{H_i\}_{i \in \Omega}$, and let $\Sigma'$ be a finite
generating set of $G$. 

Then there exist $\lambda \ge 1$, $c \ge 0$ and a finite subset
$\cH'$ of $\cH$ such that, for every finite, ordered generating set $\Sigma$ of
$G$ with $\Sigma' \cup \cH' \subseteq \Sigma \subseteq \Sigma' \cup \cH$ 
for which each $H_i$ has a geodesic biautomatic structure over
$\Sigma \cap H_i$, we have:
\begin{mylist}
\item[(i)]
$G$ has a geodesic biautomatic structure over $\Sigma$ which is
a shortlex structure if the structures on $H_i$ are shortlex;
\item[(ii)]
there exists a finite set $\Phi$ of non-geodesic words over $\Sigma$ such
that, for each word $w \in \Sigma^*$ with no parabolic shortenings and no
subwords in $\Phi$, the word $\hat{w} \in (\Sigma \cup \cH)^*$ is a
$(\lambda,c)$--quasigeodesic without vertex backtracking.
\end{mylist}
\end{proposition}

For the remainder of this paper, given a relatively
hyperbolic group $G$ and its finite generating set $\Sigma$,
we shall say that $(G,\Sigma)$ is
{\em suitable for parabolic geodesic biautomaticity} if 
\begin{mylist}
\item[(i)] $H_i \cap H_j = \{1\}$ for all $i \ne j$;
\item[(ii)] each parabolic subgroup $H_i$ has a geodesic biautomatic structure
over $\Sigma \cap H_i$, and there exist $\lambda$ and $c$ such that
the conclusions (i) and (ii) of Proposition \ref{prop:genset} hold for
$G$ and $\Sigma$.
\end{mylist}

So this applies in particular when the parabolic subgroups $H_i$ are free
abelian, as in the hypothesis of our main result.

We are now in a position to state and prove a corollary to the above proposition.

\begin{corollary}\label{cor:genset}
Let $(G,\Sigma)$ be relatively hyperbolic and
suitable for parabolic geodesic biautomaticity.
If  $w \in \Sigma^*$ is a geodesic word
that represents an element of $H_i$ for some $i$, then
$w \in (\Sigma \cap H_i)^*$. In particular, we have
$H_i = \langle \Sigma \cap H_i \rangle$ for each $i$.
\end{corollary}
\begin{proof}
Since $w$ is geodesic, it cannot contain subwords in $\Phi$ or have parabolic
shortenings.  The fact that $\hat{w}$ is without vertex backtracking
implies that $|\hat{w}|=1$ and hence that $w$ is a word over $\Sigma \cap
H_i$.  Since every element of $H_i$ can be represented by some geodesic word $w$,
$H_i$ is generated by $\Sigma \cap H_i$.
\end{proof}

\begin{lemma} \label{lem:qgeo}
Let $(G,\Sigma)$ be relatively hyperbolic and
suitable for parabolic geodesic biautomaticity, and
let $w \in \Sigma^*$ be a geodesic word.
Then there exists a constant $\lambda \ge 1$ such that
$\hat{w}$ labels a $(\lambda,0)$--quasigeodesic path in $\widehat{\Gamma}$
that does not vertex backtrack.
\end{lemma}
\begin{proof}
It follows from Proposition~\ref{prop:genset} that $\hat{w}$ labels
a  $(\lambda,c)$--quasigeodesic path for some $\lambda \ge 1$ and $c \ge 0$
and, since $w$ is geodesic, it cannot represent $1_G$ unless $w =\emptyword$,
and so by increasing $\lambda$ if necessary we may assume that $c=0$.
\end{proof}

Unfortunately, the biautomatic structure for $G$ that is given by
Proposition \ref{prop:genset} does not appear to have all of the properties
that we need in the proofs of our main results. For that we need a structure in
which $\hat{w}$ is a geodesic for words $w$ in the language. Since no such
structure appears in the literature, we need to establish its existence
here. It turns out that this structure could be asynchronous,
but that will be adequate for our purposes.

\begin{proposition}\label{prop:ftqd}
Let $(G,\Sigma)$ be relatively hyperbolic and
suitable for parabolic geodesic biautomaticity, and
let $\Gamma:=\Gamma(G,\Sigma)$,
$\widehat{\Gamma}:=\widehat{\Gamma}(G,\widehat{\Sigma})$.
Let $u,v, w_1,w_2$ be words over $\Sigma$ satisfying $w_1u = _G vw_2$,
with $|w_1|, |w_2| \le k$ for some $k\geq 0$
and, in quadrilaterals in $\Gamma$ and $\widehat{\Gamma}$ whose sides
are labelled by the words in that equation,
let $p$ and $\hat{p}$ be paths (in $\Gamma$ and  $\widehat{\Gamma}$) labelled by $u$ and $\hat{u}$,
and let $q$ and $\hat{q}$ be the paths labelled by $v$ and $\hat{v}$.

Suppose that
\begin{mylist}
\item[(a)] for each parabolic subgroup $H_i$, all $H_i$-components of both
$u$ and $v$ lie in the specified geodesic biautomatic structure;
\item[(b)] for some $\lambda \ge 1$ and $c \ge 0$,
the paths $\hat{p}$ and $\hat{q}$
are $(\lambda,c)$-quasigeodesics that do not backtrack.
\end{mylist}
Then
\begin{mylist}
\item[(i)] there is a constant $\ee' = \ee'(\lambda,c,k)$ such
that the paths $p$ and $q$ $\ee'$-fellow travel in $\Gamma$, 
in such a way that those vertices of $p$ that are also vertices of $\hat{p}$
have at least one corresponding vertex on $q$ that is also a vertex of
$\hat{q}$, and vice versa;
\item[(ii)] there is a constant $k' = k'(\lambda,c,k)$ such that, whenever two
vertices $b_1$ and $b_2$ on $q$ (or $p$) are both at $\Gamma$-distance at most
$\ee'$ from the same vertex on $p$ (or $q$), then the distance in the path $q$
(or $p$) between $b_1$ and $b_2$ is at most $k'$;
\item[(iii)]  we have $|u| \le (k'+1)|v|$ and $|v| \le (k'+1)|u|$.
\end{mylist}
\end{proposition}
\begin{proof}
Applying Proposition~\ref{prop:bcpp} to the paths $\hat{p}$ and $\hat{q}$, we 
choose $\ee_1 \geq \ee(\lambda,c,k)$ such that also $\ee_1 \ge k$, 
and then choose $\ee_2 \geq \ee(\lambda,c,\ee_1)$ such that $\ee_2 \ge \ee_1$.
For a subword $u[i:j)$ of $u$, we denote by $p[i:j)$ the subpath of
$p$ labelled by $u[i:j)$.  So, if $u[i:j)$ is a component of $u$, then
$p[i:j)$ is a component of $p$.

Let $u[i_1:j_1)$,\,$u[i_2:j_2),$\,$\ldots,$\,$ u[i_t:j_t)$ with
$i_1 < i_2 < \cdots < i_t$ be the components of $u$ of length greater that
$\ee_2$, and let the corresponding $\ee_1$-similar components of $v$
to which they are connected be
$v[k_1:l_1)$,\,$v[k_2:l_2)$,\,$\ldots$,\,$v[k_t:l_t)$.
(Since $\hat{v}$ does not backtrack, these components are uniquely defined.)

\begin{figure}
\begin{picture}(340,85)(-25,-20)
\put(0,15){\circle*{4}}
\put(0,35){\circle*{4}}
\put(25,50){\circle*{4}}
\put(25,0){\circle*{4}}
\put(85,50){\circle*{4}}
\put(85,0){\circle*{4}}
\put(105,50){\circle*{4}}
\put(105,0){\circle*{4}}
\put(165,50){\circle*{4}}
\put(165,0){\circle*{4}}
\put(185,50){\circle*{2}}
\put(195,50){\circle*{2}}
\put(205,50){\circle*{2}}
\put(185,0){\circle*{2}}
\put(195,0){\circle*{2}}
\put(205,0){\circle*{2}}
\put(225,50){\circle*{4}}
\put(225,0){\circle*{4}}
\put(285,50){\circle*{4}}
\put(285,0){\circle*{4}}
\put(310,35){\circle*{4}}
\put(310,15){\circle*{4}}

%start with horizontal arrows and labels
\put(20,-8){\vector(1,0){137}}
\put(20,27){\oval(40,70)[bl]}
\put(290,15){\oval(40,45)[br]}
\put(155,-8){\line(1,0){136}}
\put(155,-16){$u$}
\put(20,58){\vector(1,0){137}}
\put(20,28){\oval(40,60)[tl]}
\put(155,58){\line(1,0){136}}
\put(290,35){\oval(40,45)[tr]}
\put(155,63){$v$}

\put(0,15){\vector(5,-3){15}}
\put(13,7.1){\line(5,-3){12}}
\put(0,35){\vector(5,3){15}}
\put(13,42.9){\line(5,3){12}}
\put(25,50){\vector(1,0){32}}
\put(55,50){\line(1,0){30}}
\put(25,0){\vector(1,0){32}}
\put(55,0){\line(1,0){30}}
\put(85,50){\vector(1,0){12}}
\put(95,50){\line(1,0){10}}
\put(85,0){\vector(1,0){12}}
\put(95,0){\line(1,0){10}}
\put(105,50){\vector(1,0){32}}
\put(135,50){\line(1,0){30}}
\put(105,0){\vector(1,0){32}}
\put(135,0){\line(1,0){30}}
\put(165,50){\line(1,0){10}}
\put(165,0){\line(1,0){10}}
\put(215,50){\line(1,0){10}}
\put(215,0){\line(1,0){10}}
\put(225,50){\vector(1,0){32}}
\put(255,50){\line(1,0){30}}
\put(225,0){\vector(1,0){32}}
\put(255,0){\line(1,0){30}}
\put(285,50){\vector(5,-3){15}}
\put(298,42.1){\line(5,-3){12}}
\put(285,0){\vector(5,3){15}}
\put(298,7.9){\line(5,3){12}}

\put(35,40){$v[k_1:l_1)$}
\put(115,40){$v[k_2:l_2)$}
\put(235,40){$v[k_t:l_t)$}
\put(35,4){$u[i_1:j_1)$}
\put(115,4){$u[i_2:j_2)$}
\put(235,4){$u[i_t:j_t)$}

%vertical arrows and labels
\put(0,35){\vector(0,-1){12}}
\put(0,25){\line(0,-1){10}}
\put(25,50){\vector(0,-1){27}}
\put(25,25){\line(0,-1){25}}
\put(85,50){\vector(0,-1){27}}
\put(85,25){\line(0,-1){25}}
\put(105,50){\vector(0,-1){27}}
\put(105,25){\line(0,-1){25}}
\put(165,50){\vector(0,-1){27}}
\put(165,25){\line(0,-1){25}}
\put(225,50){\vector(0,-1){27}}
\put(225,25){\line(0,-1){25}}
\put(285,50){\vector(0,-1){27}}
\put(285,25){\line(0,-1){25}}
\put(310,35){\vector(0,-1){12}}
\put(310,25){\line(0,-1){10}}

\put(-15,23){$w_1$}
\put(27,23){\footnotesize{$\leqslant\!\ee_1$}}
\put(69,23){\footnotesize{$\leqslant\!\ee_1$}}
\put(107,23){\footnotesize{$\leqslant\!\ee_1$}}
\put(149,23){\footnotesize{$\leqslant\!\ee_1$}}
\put(227,23){\footnotesize{$\leqslant\!\ee_1$}}
\put(269,23){\footnotesize{$\leqslant\!\ee_1$}}
\put(313,23){$w_2$}

\end{picture}
\caption{Fellow travelling quasigeodesics}
\label{fig:redshort}
\end{figure}
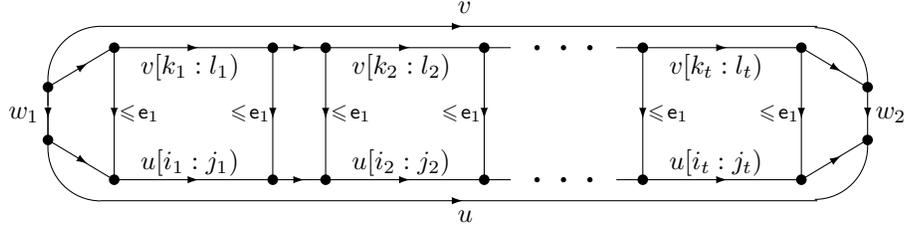

We claim that $k_1 < k_2 < \cdots < k_t$; that is, the connected components on
$v$ occur in the same order as their corresponding components on $u$.
We shall prove by induction on $r$ that $k_1 < k_2 <\cdots < k_r$ for all
$1 \le r \le t$. It is vacuously true for $i=1$, so suppose that it is true for
some value of $r$.  Since $\ee_1 \ge k$, the subpaths $\widehat{p[j_r:)}$
and $\widehat{q[l_r:)}$ of $\hat{p}$ and $\hat{q}$ are $\ee_1$-similar
$(\lambda,c)$-quasigeodesics that do not backtrack and, by choice of
$\ee_2$, the component $u[i_{r+1}:j_{r+1})$ of $u$ is connected to
a component of $v[l_r:)$. But since $\hat{v}$ does not backtrack, this
component must be $v[k_{r+1}:l_{r+1})$, so $k_r < l_r \le k_{r+1}$, which
completes the induction and establishes the claim.
See Figure \ref{fig:redshort}.

Since the component subwords of $u$ and $v$ lie in a geodesic biautomatic
structure, the paths $p[i_r:j_r)$ and $q[i_r:j_r)$ must $\ee_1L_1$-fellow
travel for $1 \le r \le t$, where $L_1$ is the maximum of the fellow
travelling constants for the biautomatic structures of the
parabolic subgroups $H_i$.

On the other hand, the subpaths $\widehat{p'}$ and $\widehat{q'}$ of
$\hat{p}$ and $\hat{q}$, where $p':=p[j_{r-1}:i_r)$ and
$q':=q[l_{r-1}:k_r)$, are
$\ee_1$-similar for $1 \le r \le t+1$ (where for convenience we define
$j_0=k_0=0$, $i_{t+1} = |u|$ and $k_{t+1} = |v|$),
so they are contained in closed $\ee_2$-neighborhoods of each other.
Now the components of $u[j_{r-1}:i_r)$ have length at most
$\ee_2$, and those of $v[l_{r-1}:k_r)$ have length at most
$\ee_2' :=\ee_2+ 2\ee_1$ (since otherwise they would be
connected to a component of $u[j_{r-1}:i_r)$, which would have length greater
than $\ee_2$). So the paths $p[j_{r-1}:i_r)$ and $q[l_{r-1}:k_r)$
are $(\lambda\ee_2',c\ee_2')$-quasigeodesics that
lie within $(\ee_2+\ee_2')$-neighborhoods of each other.
By the argument used in the proof of \cite[Proposition 3.1]{HR}, it follows
that these paths $L_2$-fellow travel for some constant $L_2$ that depends
only on $\lambda$, $k$ and $c$.
Furthermore, by increasing $L_2$ by at most $\ee_2$ we can ensure that
all vertices on either of these paths that are also vertices of $\hat{p}$
or $\hat{q}$ have at least one corresponding vertex with the same
property. So we have proved (i) with $\ee' = \max(\ee_1 L_1,L_2)$.

Now let $b_1$ and $b_2$ be vertices on $q$ as in (ii), and assume that
$b_2$ is further along $q$ than $b_1$.
Then $d_\Gamma(b_1,b_2) \le 2 \ee'$.
Let $v' \in \Sigma^*$ be a geodesic word with $v' =_G v$, and let
$q'$ and $\widehat{q'}$ be the paths in $\Gamma$ and $\widehat{\Gamma}$
labelled by $v'$ and $\widehat{v'}$, respectively.
Then $\widehat{q'}$ is quasigeodesic by Lemma~\ref{lem:qgeo},
and by applying (i) to $v$ and $v'$ with suitable constants, we
find that $v$ and $v'$ $L_3$-fellow travel for some constant
$L_3$. Let $b_1'$ and $b_2'$ be vertices of $q$ that
correspond to $b_1$ and $b_2$ in the fellow-travelling.  Then
$d_\Gamma(b_1',b_2') = d_{q'}(b_1',b_2') \le 2(\ee' + L_3)$.
Now any two vertices $c_1$ and $c_2$ of $q$ that lie between $b_1$ and
$b_2$ are each at distance at most $L_3$ from some vertex of
$q'$ that lies between $b_1'$ and $b_2'$, so we have
$d_\Gamma(c_1,c_2) \le 2\ee' + 4 L_3$. So, in particular, since the
components of $q$ are geodesics in $\Gamma$, the subpath of $q$ between $b_1$
and $b_2$ cannot have any components of length greater than $2\ee' + 4 L_3$.
So this subpath of $q$ is a quasigeodesic for suitable constants,
and (ii) now follows.

We get a corresponding result with $p$ and $q$ interchanged, and if this
results in a larger constant $k'$, then we replace $k'$ by this larger value.
Now (iii) follows directly from (ii).
\end{proof}

\begin{corollary}\label{cor:wdfsa}
Let $(G,\Sigma)$ be relatively hyperbolic and
suitable for parabolic geodesic biautomaticity.
Let $g_1,g_2 \in G$ and let $\lambda \ge 1$ and $c \ge 0$. Then there
is a deterministic 2-tape finite state automaton
$\cA = \cA(\lambda,c,g_1,g_2)$ such that:
\begin{mylist}
\item[(i)] for all $u,v$ with $(u,v) \in L(\cA)$, we have $g_1u =_G vg_2$;
\item[(ii)] for all $u, v \in \Sigma^*$ with  $g_1u =_G vg_2$ that satisfy
hypotheses (a) and (b) in the statement of Proposition~\ref{prop:ftqd},
we have $(u,v) \in L(\cA)$. 
\end{mylist}
\end{corollary}
\begin{proof}
Let $w_1,w_2 \in \Sigma^*$ be geodesic words defining $g_1$ and $g_2$, and
let $k = \max(|w_1|,|w_2|)$, and let $\ee' = \ee'(\lambda,c,k)$
and $k' = k'(\lambda,c,k)$ be the constants defined in
Proposition~\ref{prop:ftqd}\,(i) and (ii). We define a 2-tape
asynchronous {\em word-difference} automaton $\cA'$, following the
recipe of \cite[Definition 2.3.3]{ECHLPT}. The states $\cA'$ are elements of
$G$ of $\Sigma$-length at most $\ee'$, with start state $g_1$ and
single accepting state $g_2$, and transitions $g^{(x,y)} = x^{-1}gy$
for all $x,y \in \Sigma \cup \{\emptyword\}$ such that
$g$ and $x^{-1}gy$ are both states of $\cA'$. Then $\cA'$ clearly
satisfies (i) and it satisfies (ii) by Proposition~\ref{prop:ftqd}\,(i). 
But since $\cA'$ contains $\emptyword$-transitions, it is non-deterministic.

Proposition~\ref{prop:ftqd}\,(ii) enables us to define a
deterministic 2-tape automaton $\cA$ with $L(\cA) = L(\cA')$.
After reading a prefix $u_1$ of $u$, we know that the lengths
of the prefixes $v_1$ of $v$ for which $v_1^{-1}g_1u_1$ has $\Sigma$-length
at most $\ee'$ can differ by at most $k'$. So $\cA$ can read the
longest such prefix $v_1$ and remember all shorter such prefixes.
After reading each letter of $u_1$, $\cA$ need never read forward more than
$k'$ letters of $v_2$ to maintain this information. If at any time
there are no eligible prefixes $v_1$ of $v$ then $\cA$ stops and rejects
$(u,v)$. 
\end{proof}

\begin{proposition}\label{prop:automatic}
Let $(G,\Sigma)$ be relatively hyperbolic, and
suitable for parabolic geodesic biautomaticity.
Define $\cL$ to be the set of all words $u \in \Sigma^*$ such that
the derived word $\hat{u} \in \widehat{\Sigma}^*$ is geodesic
and such that, for each parabolic subgroup $H_i$, all $H_i$-components of $u$
lie in the specified geodesic biautomatic structure.
Then $\cL$ is the language of an asynchronous biautomatic structure for $G$. 

Furthermore, for any ordering of $\Sigma$ with associated lexicographical
ordering $\le_\lex$ of $\Sigma^*$, the language
$\cL_0 = \{  u \in \cL : u =_G v, v \in \cL \Rightarrow u \le_\lex v \}$
is an asynchronous biautomatic structure for $G$ with uniqueness.
\end{proposition}
\begin{proof}
If we can prove that $\cL$ is a regular language, then the
first claim will follow immediately from Corollary~\ref{cor:wdfsa}.
(The hypothesis that $\hat{u}$ is geodesic for $u \in \cL$ implies
that $\hat{u}$ does not backtrack.)
The required property that all $H_i$-components of words in $\cL$ lie
in the specified geodesic biautomatic structure is certainly testable by
a finite state automaton, so we may restrict our attention to
words that satisfy it. 

For all $u \in \cL$, we have $u_1 \in \cL$ for all prefixes $\widehat{u_1}$ of
$\hat{u}$.
So if $u \in \Sigma^*$ with $u \not\in \cL$, then $\hat{u}$ has a shortest
prefix $u_1$ such that $\widehat{u_1}$ is a prefix of $\hat{u}$
(i.e. such that $u_1$ is a union of complete components of $u$ and
subwords in $(\Sigma \setminus \cH)^*$) and $u_1 \not\in \cL$.
Since the maximal proper prefix of $\widehat{u_1}$ is geodesic,
$\widehat{u_1}$ labels a
$(1,2)$-quasigeodesic in $\widehat{\Gamma}$, and it is not
hard to see that $\widehat{u_1}$ cannot backtrack.
Let $v_1 \in \cL$ with $v_1 =_G u_1$. Then $(u_1,v_1) \in L(\cA)$, where
$\cA := \cA(1,2,1_G,1_G)$ as defined in Corollary~\ref{cor:wdfsa}.
Furthermore, if $(u_2,v_2)$ is a prefix of $(u_1,v_1)$ in an accepting path of
$(u_1,v_1)$ through $\cA$, then the $\Sigma$-length and hence
also the $\widehat{\Sigma}$-length of the element of $G$ defined by
$u_2^{-1}v_2$ is bounded. So $\cA$ can be modified to keep track of
the difference between the $\widehat{\Sigma}$-lengths of these
pairs of prefixes $(u_2,v_2)$, and hence it can detect that
$|\widehat{u_1}| > |\widehat{v_1}|$.

So a word $u \in \Sigma^*$ lies in $\cL$ if and only if its components satisfy
the required condition, and if it has no prefix $u_1$ consisting of a union of
complete components of $u$ and subwords in $(\Sigma \setminus \cH)^*$
for which there exists
$v_1$ with $(u_1,v_1) \in L(\cA)$ and $|\widehat{u_1}| > |\widehat{v_1}|$.
So $\cL$  is regular by \cite[Lemma 7.1.5]{ECHLPT}.

The final statement follows immediately from \cite[Theorem 7.3.2]{ECHLPT}
(together with its proof, in which  $\cL_0$ is defined as we have done
so here).
\end{proof}

\begin{proposition}\label{prop:nfwdlenbd}
There is a constant $\DD$ (depending on $G$ and $\Sigma$) such that
$|\nf(w)| \le \DD |w|$ for any $w \in \Sigma^*$.
\end{proposition}
\begin{proof} It is sufficient to prove this when $w$ is a geodesic word.
Then by Lemma~\ref{lem:qgeo} $\hat{w}$ is a $(\lambda,0)$-quasigeodesic
for some $\lambda \ge 1$, and the result follows by applying
Proposition~\ref{prop:ftqd}\,(iii) with $u=w$ and $v=\nf(w)$ and
$w_1=w_2 = \emptyword$.
\end{proof}

\begin{remark}\label{rem:subclosure}
The languages $\cL$ and $\cL_0$ are not necessarily closed under subwords,
but if $u \in \cL$ or $u \in \cL_0$, and $u_1$ is a subword that contains only
complete components of $u$ together with subwords in $\Sigma \setminus \cH)^*$, then $u_1 \in \cL$ or $\cL_0$.
\end{remark}

\begin{remark} It is not difficult to show that the paths in $\Gamma$ labelled
by words in $\cL$ and $\cL_0$ are quasigeodesics, but we shall not need that
property.
\end{remark}

\section{Background on straight line programs}
\label{sec:slp_background}

Let $\cG=(V,S,\rho)$ be a straight line program over an alphabet $\Sigma$.
For a variable $A \in V$,
the word $\rho(A)$ is called the {\em right-hand side} of $A$.
We define the {\em size} of $\cG$ to be the total length
of all right-hand sides: $|\cG| := \sum_{A \in V} |\rho(A)|$.

It %then 
follows from the acyclicity condition that
each variable $A$ in $V$ has a well defined {\em  height}, $\Ht(A)$,
namely the smallest positive integer $r$ for which $\rho^r(A) \in \Sigma^*$. 
By removing from $\cG$  all variables that do not occur in any image
$\rho^k(S)$ for $k \in \N$, we obtain (in 
time that is a polynomial function of $|\cG|$) an \slp in which 
$S$ is the only variable of maximal height.
We call such an \slp\ {\em trimmed};
we will often want to assume this property for an \slp.

Suppose that $A$ is a variable of height $r$ in $\cG$.
We  define $\val_{\cG}(A):=\rho^r(A) \in \Sigma^*$,
and observe that $\val_{\cG}(S)= \val(\cG)$.
We also define a (trimmed) \slp $\cG_A := (V_A,A,\rho_A)$ over $\Sigma$,
the {\em restriction} of $\cG$ to $A$,
with start variable $A$, set of variables $V_A$ consisting of all $B \in V$
that appear within $\rho^k(A)$ for some $k \geq 0$,
and map $\rho_A$ defined to be the restriction of $\rho$ to $V_A$.
We note that for any $B \in V_A$, $\val_\cG(B)=\val_{\cG_A}(B)$,
and in particular $\val(\cG_A) = \val_{\cG}(A)$.

If every right-hand side (except that of $\rho(S)=\emptyword$, if it
occurs) has the form $a \in \Sigma$ or $BC$ with $B,C \in V$
(that is, the grammar associated with $\cG$ is in Chomsky normal form),
then we shall say that $\cG$ itself is in {\em Chomsky normal form}. 
We shall often want to assume this property.

Given \slps $\cG_1$ and $\cG_2$ over a single alphabet $\Sigma$,
we denote by $\cG_1\cG_2$ the
\slp with $\val(\cG_1\cG_2)$ equal to the concatenation
$\val(\cG_1)\val(\cG_2)$,
which we derive from $\cG_1$ and  $\cG_2$ by adding to their disjoint union a
single variable $S_{\cG_1\cG_2}$ and the single production
$S_{\cG_1\cG_2}\rightarrow S_{\cG_1}S_{\cG_2}$.
We extend this definition to a concatenation of any number of \slps,
and also to a concatenation of \slps and words over $\Sigma$,
so that a concatenation of the form 
$\cG = u_0\cG_1u_1\cdots \cG_ku_k$, where each $\cG_i$ is a \slp and each $u_i$
a possibly empty word over $\Sigma$ is formed by the addition 
of a single production 
\[ S_\cG \rightarrow u_0S_{\cG_1}u_1S_{\cG_2}u_2\cdots S_{\cG_k}u_k \]
to the disjoint union of the productions of the \cslps $\cG_i$.

We will make use of a number of results from the literature,
which we collect together here as a single proposition:
\begin{proposition}
\label{prop:slp_results}
Let $\cG=(V,S,\rho)$ be a straight line program over a finite alphabet $\Sigma$.
\begin{mylist}
\item[(i)] An  algorithm exists that transforms $\cG$ 
%\comment{DFH: Added condition $|\cG'| \le 2|\cG|$}
into an \slp $\cG'$ in Chomsky normal form with $|\cG'| \le 2|\cG|$ and
$\val(\cG) = \val(\cG')$, in time that is a linear function of $|\cG|$; see for example \cite[Proposition 3.8]{Loh14}.
\item[(ii)] We have $|\val(\cG)| \leq  3^{|\cG|/3}$
   \cite[proof of Lemma 1]{CLLLPPSS05}.
\item[(iii)] The length $|\val(\cG)|$ can be computed in 
time that is a polynomial function of $|\cG|$
\cite[Proposition 3.9]{Loh14}.
\item[(iv)] For $0 \leq i \leq j \leq  |\val(\cG)|$,
an \slp with value $\val(\cG)[i:j)$ can be computed in 
time that is a polynomial function of $|\cG|$
\cite[Proposition 3.9]{Loh14}.
\item[(v)] 
Given a deterministic finite state automaton $M$ over the alphabet
$\Sigma$ and an \slp $\cG$ over the alphabet $\Sigma$, it can be determined in
time that is a polynomial function of $|\cG|$
whether $\val(\cG) \in L(M)$  \cite[Theorem 3.11]{Loh14}.
\item[(vi)]
Given two \slps $\cG$ and $\cH$, it can be checked in  
time that is a polynomial function of $|\cG|+|\cH|$
whether $\val(\cG) = \val(\cH)$ \cite{Pla94}.
\end{mylist}
\end{proposition}

\begin{definition}\label{def:root}
Let $\cG=(V,S,\rho)$ be an \slp over an alphabet $\Sigma$, with value
$w_1uw_2$.  We say that $u$
has a {\em root}, $A$  in $V$, if for some $k$ and $\ell=\Ht(S)-k$,
we have $\rho^k(S)=\alpha A \beta$, 
where $\alpha,\beta \in (V \cup \Sigma)^*$, 
$\rho^\ell(\alpha)=w_1$, $\rho^\ell(\beta)=w_2$, and $\rho^\ell(A)=u$.
\end{definition}

%%%%% CSLPs and TCSLPs
\subsection{ Extensions of \slps}\label{subsec:extend_slps}
It will sometimes be convenient in our proofs
to make use of particular extensions of \slps, namely cut-\slps and
tethered-\slps.

Cut-\slps (which we shall abbreviate as \cslps) are
defined analogously to
\slps, but in addition a cut-\slp may contain right-hand sides that are written
in the form $B[:i)$ or $B[i:)$ for a variable $B$ and an integer $i \geq 0$
\cite{Loh14}; \cslps are used in
the construction of a polynomial
time algorithm for the compressed word problem of a free group in
\cite{Loh06siam}, where they are called {\em composition systems}.
We call $[:i)$ and $[i:)$ cut operators.
It is convenient to allow cut operators of the form $B[i:j)$ for
$0 \le i < j$, which we can achieve as the combination of two cut operators
$B[:j)[i:)$.
If $\rho(A) = B[i:j)$ in a \cslp $\cG$, then we define
$\val_{\cG}(A)$ to be the string $\val_{\cG}(B)[i:j)$.

Given a \cslp $\cG$, we denote by $\cG [i:j)$ the \cslp with value
$\val(\cG)[i:j)$ that we derive from $\cG$ by adding to $\cG$ a single
variable $S_{\cG[i:j)}$ (the start variable of $\cG [i:j)$)
and the single production
$S_{\cG[i:j)}\rightarrow S_{\cG}$ to the \cslp $\cG$.

In a \cslp, as in a \slp, we require that the associated binary relation
is acyclic.
We define concatenations of \cslps analogously to concatenations of \slps.
The following results can be found in the literature;
the second follows from the first together with Proposition~\ref{prop:slp_results}\,(vi).\\
% The linebreak is there to repair a problem with vertical space.

\begin{proposition}\label{prop:cslps}

\begin{mylist}
\item[(i)]
From a given \cslp $\cG$ one can compute, in 
time that is a polynomial function of $|\cG|$,
an \slp $\cG'$
such that $\val(\cG) = \val(\cG')$
\cite{Hag00}; see also \cite[Theorem 3.14]{Loh14}.
\item[(ii)] Given {\cslp}s $\cG_1$ and $\cG_2$,
it can be checked in polynomial time whether $\val(\cG_1) = \val(\cG_2)$
\end{mylist}
\end{proposition}

As in \cite{HLS}, the proof of the main theorem also involves extensions
to \slps and \cslps that we call \tslps and \tcslps,
respectively (T stands for ``tethered'').
These extensions only make sense when the alphabet $\Sigma$ is the
(inverse closed) generating set of a group $G$, and when
$G$ is equipped with a normal form $\nf(w)$ for words $w \in \Sigma^*$;
that is, $\nf(w) \in \Sigma^*$, $\nf(w) =_G w$, and, for $v,w \in \Sigma^*$,
$v=_Gw \Rightarrow \nf(v) = \nf(w)$.

We extend the definition of a \slp or \cslp to that of a
\tslp or \tcslp over such an alphabet $\Sigma$
by allowing additional right-hand sides
that are expressions of the form $B\langle \alpha,\beta \rangle$
with  a variable $B$, and strings $\alpha,\beta$ over $\Sigma$, each of length
at most $\JJ$,
where  $\JJ = \JJ_\cT \in \N$ is a parameter of the \tslp or \tcslp $\cT$.
Given a right-hand side $\rho(A)=B\langle \alpha,\beta \rangle$ in a
\tslp or \tcslp $\cT$, we define
\[\val_{\cT}(A) := \nf(\alpha \, \val_{\cT}(B) \, \beta^{-1}).\]
In \tslps and  \tcslps, as with \slps and \cslps, we require that the
associated binary
relations are acyclic, and we define concatenations of \tslps and \tcslps
analogously to concatenations of \slps.

We define the {\em size} of a variable $A$ in a \tcslp $\cT$ to be the total
number of occurrences of symbols from $\Sigma \cup V$ in $\rho(A)$, and
the size of $\cT$ is obtained by taking the sum over all variables.

As we did for \slps, for a \tcslp $\cU$ over $\Sigma$, we
define the height, $\Ht(A)$, of a variable $A$ of $\cU$
to be the smallest integer $k$ such that $\rho_{\cU}^k(A) \in \Sigma^*$.
Proposition~\ref{prop:slp_results}\,(i) extends to \tcslps: for a
given \tcslp $\cU$  over $\Sigma$ with $\val(\cU) \ne \epsilon$, we can in
linear time construct a \tcslp $\cU'$ with the same value in
%\comment{DFH: Added condition $|\cU'| \le 2|\cU|$}
which $|\cU'| \le 2 |\cU|$ and all right hand sides $\rho_{\cU'}(A)$ that lie in
$(V_{\cU'} \cup \Sigma)^*$ have the
form $a \in \Sigma$ or $BC$ with $B,C \in V_{\cU'}$.

We say that the \tcslp (or \tslp) $\cT$ 
is {\em $\nf$-reduced} if for every variable
$A$ the word $\val_{\cT}(A)$ is $\nf$-reduced; that is
$\nf(\val_{\cT}(A)) = \val_{\cT}(A)$.

In Section~\ref{sec:convert} we shall prove Theorem~\ref{thm:convert}
that, given a \tcslp $\cT$
over a suitably chosen generating set $\Sigma$ of a group that is hyperbolic
relative to a collection of free abelian subgroups
(and which satisfies suitable conditions relating to
the `splitting of components', explained in Definition~\ref{def:split}),
we can, in time that is a polynomial function of $|\cT|$
(depending on $\JJ_\cT$),  
compute an \slp $\cG$ with $\val(\cT) = \val(\cG)$.
This result will be a vital component of our main result. Its proof will
require the following generalisation of
Proposition~\ref{prop:slp_results}\,(ii) to $\nf$-reduced \tcslp\,s.

\begin{proposition}\label{prop:tcslplenbd} Let $\cT=(V,S,\rho)$ be an
$\nf$-reduced \tcslp. Then there is a constant $\JJ'$ (depending on
$\Sigma$ and $\JJ_\cT$) such that
$|\val(\cT)| \leq  (\JJ')^{|\cT|}$
\end{proposition}
%\comment{DFH: rewrote proof}
\begin{proof}
As we saw above, there is a \tcslp $\cT' = (V',S,\rho')$ with $\val(\cT') =
\val(\cT)$, $|\cT'| \le 2|\cT|$ in which all right hand sides
$\rho'(A)$ that lie in $(V' \cup \Sigma)^*$ have the
form $a \in \Sigma$ or $BC$ with $B,C \in V'$.
If $\rho(A) = B\langle \alpha,\beta \rangle$ with
$|\alpha|,|\beta| \le J_\cT$ then $\rho'(A) = \rho(A)$ and,
by \ref{prop:ftqd}\,(iii), we have
$|\val(A)| \le  (k'+1)\val(B)|$ with $k' = k'(1,0,\JJ_\cT)$.

We claim that $\val_{\cT'}(A) \le \max(2,k'+1)^{\Ht(A)}$ for all $A \in V'$.
Since $\Ht(S) \le |\cT'| \le 2|\cT|$ this will prove the proposition with
$\JJ' = \max(2,k'+1)^2$.

The proof of the claim is by induction on $\Ht(A)$, and the base case $\Ht(A)=0$
is clear. Otherwise, if $\rho'(A) = BC$ then $|\val_{\cT'}(A)| =
|\val_{\cT'}(B)|+|\val_{\cT'}(C)|$ with $\Ht(B),\Ht(C) \le \Ht(A)$;
if $\rho'(A) = B[i:)$ or $B[:j)$, then $|\val_{\cT'}(A)| \le |\val_{\cT'}(B)|$;
and if $\rho'(A) = B\langle \alpha,\beta \rangle$ with
$|\alpha|,|\beta| \le J_\cT$ then, as saw above,
$|\val_{\cT'}(A)| \le  (k'+1)|\val_{\cT'}(B)|$. In all three cases the claim
follows immediately from the inductive hypothesis.
\end{proof}

\section{Some constructions of \slps for abelian  groups}
\label{sec:slps_ab}

We need a few results about \slps for finitely generated abelian groups,
which will be the parabolic subgroups of our relatively hyperbolic groups.

\begin{lemma}
\label{lem:ab_polytime}
Let $G$ be a finitely generated abelian group with generating set 
$X =\{x_1^{\pm 1},\ldots,x_k^{\pm 1}\}$, and let $\cG$ be an \slp over $X$.
Then, in time that is polynomial in $|\cG|$, we can
\begin{mylist}
\item[(i)]
compute the vector $(n_1,n_2,\cdots,n_k)$,
defined by $\val(\cG) =_G \prod_{i=1}^t x_i^{n_i}$ (where the integers
$n_i$ are output as binary numbers),
\item[(ii)]
construct an \slp  $\cG'$ with $\val(\cG') = \prod_{i=1}^t x_i^{n_i}$
and $|\cG'| \le \max(4k\log_2(|\val(\cG')|),1)$.
\end{mylist}
\end{lemma}
\begin{proof}
The fact that we can compute the integers $n_i$ in polynomial time follows from
Proposition \ref{prop:slp_results}\,(ii) together with the fact that
we can perform addition and subtraction on integers of absolute value at most
$N$ in time $O(\log N)$.

For part (ii), note first that for any $x \in X$ and integer $n>0$,
we can define an \slp with value $x^n$ and size at most $\max(3 \log_2(n),1)$
by expressing $n$ in binary and introducing a variable $A_i$ with size 2 
and value $x^{2^i}$ for each $i$ with $2^i< n$.
For example, with $n=14$, we write $14 = 2^1 + 2^2 + 2^3$, and define the
\slp $(\{S,A_1,A_2,A_3\},S,\rho)$ with $\rho(S)=A_1A_2A_3$,
$\rho(A_1)=xx$, $\rho(A_2)=A_1A_1$, $\rho(A_3)=A_2A_2$, which has value $x^{14}$
and size 9. So we can construct an \slp $\cG'$ with value $\val(\cG)$
and size at most
$(3\log_2(|\val(\cG')|)+1)k \le \max(4k\log_2(|\val(\cG')|),1)$.
\end{proof}

\begin{proposition}\label{prop:abslp}
Let $G = \langle X \rangle$ be a finitely generated abelian group.
Then $X$ is contained in a finite subset $Y$ of $G$ for which,
given an \slp $\cG$ over $Y$, we can
construct an \slp $\cG'$ over $Y$ with
$\val(\cG') = \slex_Y(\val(\cG))$ and
$|\cG'| \le \max(4|Y|\log_2(|\val(\cG')|),1)$
in time that is a polynomial function of $|\cG|$.
\end{proposition}

\begin{proof}
Choose generators $z_1,\ldots,z_r,z_{r+1},\ldots,z_s$ of the cyclic direct
factors of $G$, where $z_i$ has infinite order if and only if $i \le r$.
We can write each element $x$ of $X$ as
$\prod_{i=1}^s z_i^{{n_i(x)}}$ with ${n_i(x)} \in \Z$;
note that the set of all ${n_i(x)}$ is finite, and so its elements can be
regarded as constants.
Let $e$ be the exponent of the torsion subgroup of $G$,
and
let $M := e(\max_{x \in X} \sum_{i=1}^r |{n_i(x)}|+1)$.

Now we define $Y$ to be the union of the three sets
$\{ z_i^{\pm M}: 1 \le i \le r\}$,
$\{ z_i^{\pm 1}: 1 \le i \le s\}$, and $X$; we order $Y$ so that
\begin{mylist}
\item[$\bullet$] pairs of mutually inverse generators are adjacent
in the ordering; and
\item[$\bullet$]
the generators $z_i^{\pm M}$ come first in the
ordering, with $z_1^M< z_1^{-M} < \cdots < z_r^M < z_r^{-M}$.
\end{mylist}
We denote by $y_i$ and $y_i^{-1}$
the elements in positions $2i-1$ and $2i$ of this ordered set of generators.
Suppose that there are $2t$ such generators in total.

We claim that, in any geodesic word $\prod_{i=1}^t  y_i^{n_i}$ with
$n_i \in \Z$, we have $|n_k| < M$ for all $k>r$.

To see this, given such a geodesic word, suppose that $r < k \le t$.
If $y_k= z_j$ for some $j$, then we could replace $y_k^M$ by
the shorter word $y_j$ if $j \le r$ and by the empty word if $j>r$;
so we must have $|n_k| < M$. Otherwise, we have $y_k \in X$, and then
$x := y_k =_G \prod_{i=1}^s z_i^{n_i(x)}$ and (because $e|M$),
$x^M =_G \prod_{i=1}^r z_i^{Mn_i(x)}=_G \prod_{i=1}^r y_i^{n_i(x)}$. 
This last word is shorter than $x^M$,
since $M > \sum_{i=1}^r |n_i(x)|$ So replacing $x^M$ by 
$\prod_{i=1}^r y_i^{n_i(x)}$ would be a reduction in length,
and hence again we must have $|n_k| < M$.

Now it follows from Lemma~\ref{lem:ab_polytime} that, in time polynomial in
$|\cG|$, we can write $g := \val(\cG)$ as an integer vector over
the generators in $Y$, and hence as a product
$\prod_{i=1}^t y_i^{n_i} \in G$. Now we want to compute $\slex(g)$.
If any $|n_k| \ge M$ for $k > r$, then we write $n_k = qM + n_k'$ with
$|n_k'| < M$ and $\sgn(n_k) = \sgn(n_k')$ and replace the expression for $g$ by
an equivalent expression in which $y_k^{qM}$ is replaced by a shorter word,
as described in the preceding paragraph.  This involves integer arithmetic and
can be done in time polynomial in the sizes of the integers.
So we may assume that $g = \prod_{i=1}^t y_i^{n_i} \in G$ with $|n_k| < M$ for
all $k>r$.

Now suppose that $\slex(g) = \prod_{i=1}^t y_i^{n_i'}$. Then
$\prod_{i=1}^t y_i^{(n_i-n_i')} =1$. and since $|n_k'|,|n_k| < M$,
we have $|n_k-n_k'| < 2M$ for all $k>r$. Since $y_1,\ldots,y_r$ are free
generators, there are no nontrivial relations in $G$ that involve only
$y_1,\ldots,y_r$, and so the number of lists of integers $m_1,\ldots,m_t$ for
which $\prod_{i=1}^t y_i^{m_i} =1$ and $|m_i| < M$ is at most $(2M)^{t-r}$.
So we can consider each of these equations in turn, and thereby find all
possible values of $n_i'$. From these, we select the shortlex least
representative of $g$.
We can now use Proposition \ref{lem:ab_polytime} to construct the
required \slp $\cG'$.
\end{proof}

\section{Examining the geometry of $\widehat{\Gamma}$}
\label{sec:gammahat_geom}

For the remainder of this article, $G$ will be a group that is hyperbolic
relative to a collection of free abelian subgroups $H_i$,
and $\cH$ the set of non-identity elements of those subgroups,
as in Section~\ref{sec:relhyp}. 

By \cite[Theorem 4.3.1]{ECHLPT}, abelian groups are shortlex automatic
with respect to any finite ordered generating set, and a little thought
shows that any automatic structure for an abelian group must be
biautomatic. So the hypotheses of shortlex biautomaticity of the
subgroups $H_i$ in Proposition~\ref{prop:genset} are satisfied for all
choices of finite, ordered generating sets of the $H_i$.

\subsection{Fixing the generating set $\Sigma$ and constant $\delta$}
\label{subsec:relhyp_param1}
Given a generating set $\Sigma'$ for $G$, we define, for each $i$,
$X_i := (\Sigma' \cup \cH') \cap H_i$, where $\cH'$ is the finite
subset of $\cH$ whose existence is guaranteed by Proposition~\ref{prop:genset}.
It follows from Corollary~\ref{cor:genset} that, for each $i$,
$X_i$ generates $H_i$.
Now we select finite subsets $Y_i \supset X_i$ as in 
Proposition~\ref{prop:abslp}, and define 
$\Sigma := \Sigma' \cup \cH' \cup \bigcup_{i\in \Omega} Y_i.$  
Our construction ensures that $\Sigma \cap H_i = Y_i$.
For the remainder of this article, we use this generating set $\Sigma$ for $G$,
and denote $\Sigma \cap H_i$ by $\Sigma_i$;
we note that $(G,\Sigma)$ is suitable for parabolic geodesic biautomaticity.
For convenience we assume that the elements of $\Sigma$ represent distinct
elements of $G$ and that no element of $\Sigma$ represents the identity element.

Furthermore, we define $\cL_0$ to be the asynchronous automatic structure of
$G$ with uniqueness that is defined in Proposition \ref{prop:automatic}, where
we use the shortlex biautomatic structure on $H_i$ over $\Sigma_i$.
We shall call words in $\cL_0$ {\it $\cL_0$-reduced} and, for $v \in \Sigma^*$,
we denote the  unique element $u \in \Sigma^*$ with $u =_G v$ by $\nf(v)$.
We shall use this normal form in all of our \tcslps over $\Sigma$.
Our assumptions on $\Sigma$ ensure that $\nf(a) = a$ for all $a \in \Sigma$.

Now that we have fixed $\Sigma$, we can also fix the associated Cayley
graphs $\Gamma$ and $\widehat{\Gamma}$.
We know from the definition of relative hyperbolicity that $\widehat{\Gamma}$
is  Gromov hyperbolic, and we fix the constant $\delta>0$ such that 
$\widehat{\Gamma}$ is a $\delta$-hyperbolic space; that is, all geodesic
triangles in  $\widehat{\Gamma}$ are $\delta$-thin (and hence also
$\delta$-slim) as defined in \cite[Chapter 1]{AL}.
We assume these choices for $\Sigma$, $\Gamma$, $\widehat{\Gamma}$ and $\delta$
for the remainder of this article.

\subsection{Properties of some geodesic triangles and quadrilaterals in $\widehat{\Gamma}$}
\label{subsec:Gammahat_triangles}

\begin{proposition}\label{prop:backup}
Let $u,v,w \in \Sigma^*$ be words over our selected generating set
for $G$ with $v =_G uw$, where $\hat{u}$ and $\hat{v}$ are geodesic
words over $\widehat{\Sigma}$, and $|w| \le \kappa$ for some $\kappa$.
Then constants $K_1(\kappa)$ and $L_1(\kappa)$ exist such that, for any vertex
$d$ on the  path in $\widehat{\Gamma}$ labelled by $\hat{u}$ and at
distance at least $K_1(\kappa)$ from the end of that path,
there is a vertex $e$ of $\widehat{\Gamma}$ on the path labelled by
$\hat{v}$, with $d_\Gamma(d,e)\leq L_1(\kappa)$.

\end{proposition}
\begin{proof}
Since $\hat{u}$ labels a geodesic path in $\widehat{\Gamma}$ and
$|\hat{w}| \le |w| \le \kappa$, we know that $\hat{u}\hat{w}$
labels a $(1,2\kappa)$--quasigeodesic path in $\widehat{\Gamma}$.
Our aim is to replace $\hat{u}\hat{w}$ by a $G$-equivalent word
$t \in \widehat{\Sigma}^*$ that also labels a
$(1,2\kappa)$--quasigeodesic path, and does not backtrack,
and whose prefix of length $|\hat{u}|-K_1(\kappa)$ matches the
corresponding length prefix of $\hat{u}$ for some constant $K_1(\kappa)$.
The result will then follow directly from Theorem~\ref{prop:bcpp}\,(1),
applied to $t$ and $\hat{v}$.

If $\hat{u}\hat{w}$ backtracks, then it contains a subword of length
greater than $1$ that represents an element of $H_i$ for some $i \in \Omega$.
Since such elements have length at most $1$ over $\widehat{\Sigma}^*$, and
$\hat{u}\hat{w}$ labels a $(1,2\kappa)$--quasigeodesic
path, any such subword has length at most $K_1(\kappa):= 1+2\kappa$.
Furthermore, since $\hat{u}$ does not vertex backtrack, such a subword
must intersect the suffix $\hat{w}$ nontrivially. So it does not intersect
the prefix $\hat{u}(:|\hat{u}|-K_1(\kappa)]$ of $u$. So, after
replacing any such subwords by $G$-equivalent words of length $1$ over
$\Gamma^*$, the resulting word $t$ does not backtrack, and has the desired
property.
\end{proof}

Our next result can be proved by two applications of
Proposition~\ref{prop:backup}, and we omit the details.

\begin{lemma}\label{lem:backup2}
Let $u,v,w_1,w_2 \in \Sigma^*$ be words over $\Sigma$ with $w_1u =_G vw_2$, and
suppose that $\hat{u}$ and $\hat{v}$ are geodesic words, and that
$|w_1|, |w_2| \le \kappa$ for some $\kappa$.
Then constants $K_2(\kappa)$ and $L_2(\kappa)$ exist, such that, 
for any vertex $d$ on the path in $\widehat{\Gamma}$ labelled by
$\hat{u}$ and at distance at least $K_2(\kappa)$ from the
beginning and the end of that path,
there is a vertex $e$ of $\widehat{\Gamma}$ on the path labelled by
$\hat{v}$, with $d_\Gamma(d,e)\leq L_2(\kappa)$.
\end{lemma}

\begin{proposition}\label{prop:backup2}
There is a linear function $\ff:\N \to \N$ and a constant $L'$,
with the following property.
Suppose that $u,v,w_1,w_2$ are words over our selected
generating set $\Sigma$ for $G$, with $w_1u =_G vw_2$, and
suppose that $\hat{u}$ and $\hat{v}$ are geodesic words.
Consider a quadrilateral in $\widehat{\Gamma}$ with sides labelled
$\widehat{w_1},\hat{u},\widehat{w_2},\hat{v}$.

Then, for any vertex $d$ on the path labelled by
$\hat{u}$ in that quadrilateral,
and at distance at least $\ff(\max\{|w_1|,|w_2|\})$
from the beginning and the end points of that path, there is a vertex $e$ of
the path labelled by $\hat{v}$ in the quadrilateral,
with $d_\Gamma(d,e)\leq L'$.
\end{proposition}

\begin{proof}
We recall the constant $\delta$ defined in Section~\ref{subsec:relhyp_param1}.
In the quadrilateral defined in the statement, replace the sides
labelled by $\widehat{w_1}$ and $\widehat{w_2}$ by geodesic paths
$p_1$ and $p_2$ between their endpoints to give a geodesic quadrilateral
in $\widehat{\Gamma}$. So $|p_i| \le |\widehat{w_i}| \le |w_i|$ for $i=1,2$.

Since $\widehat{\Gamma}$ is $\delta$-hyperbolic,
any vertex on any side of this quadrilateral is at distance at most $2\delta$ in
$\widehat{\Gamma}$ from a vertex on one of the other three sides.

Now for a vertex $d$ on the path labelled $\hat{u}$ 
that is at distance $\ell$ along $\hat{u}$ from the
start point, all vertices on the path $p_1$ are at distance
at least $\ell- |p_1| \ge \ell - |w_1|$ from $d$ in $\widehat{\Gamma}$.  So, if
$\ell \ge |w_1| + 2\delta+1$, then $d$ cannot be at distance
at most $2\delta$ from a vertex on $p_1$. Similarly,
if the distance of $d$ on $\hat{u}$ from the end point of $\hat{u}$ is
greater than $|w_2| + 2\delta+1$, then it cannot be at distance
at most $2\delta$ from a vertex on $p_2$.
So if both of those conditions on $d$ hold, then $d$ must be at distance
at most $2\delta$ in $\widehat{\Gamma}$ from a vertex $e$ on the side of the
quadrilateral labelled $\hat{v}$.

Now, by Lemma~\ref{lem:backup2}, there are constants
$K_2:=K_2(2\delta)$ and $L_2:=L_2(2\delta)$ depending only on $G$ and
$\Sigma$, such that any vertex $d$ on the path labelled $\hat{u}$ that is
at distance at least $|w_1| + 2\delta+1+K_2$ and
$|w_2| + 2\delta+1+K_2$ from the start and end points of that path,
respectively, is at distance at most $L_2$ in $\Gamma$ from a vertex on the
side of the quadrilateral labelled $\hat{v}$.
This proves the proposition with $L'=L_2$.
\end{proof}

\subsection{Fixing the constant $L$}
\label{subsec:relhyp_param2}
We define $L$ to be the larger of the constants $L_1(\delta)$  and
$L'$ that were defined in Propositions \ref{prop:backup}
and~\ref{prop:backup2}, respectively.
We will refer to $L$ repeatedly in the final two sections of the paper.

\section{Some constructions of \slps for relatively hyperbolic groups}
\label{sec:slps_relhyp_props}

The proof of our main result (which is split across Theorems
~\ref{thm:slextcslp} and~\ref{thm:convert}) will need some
technical results relating to \slps for our relatively hyperbolic group $G$
over our selected generating set $\Sigma$.
In general these results will be applied to sub-\slps of the \slp that is input
and the \slps that are derived from it within the above two theorems. 

\begin{proposition}\label{prop:comproot}
Let $\cG$ be an \slp over our selected generating set $\Sigma$ for $G$, and
let $w:=\val(\cG)$. Then, in time polynomial in $|\cG|$, we can construct an
\slp $\cG'$ in Chomsky normal form with value $w$
such that,
for all variables $A$ of $\cG'$, all components of $\val(A)$ have roots
in $\cG'_A$.

Furthermore, for each parabolic subgroup $H_i$, we can compute a list of those
variables $A$ of $\cG'$ for which $\val_{\cG'}(A) \in \Sigma_i^*$.
\end{proposition}
\begin{proof}
We may assume that $\cG$ is trimmed and in Chomsky normal form.
We construct  $\cG'=(V',S,\rho')$ from $\cG=(V,S,\rho)$ by 
modifying the map $\rho$ on some of the variables in $V$,
while also introducing some new variables that are needed within those
new images for $\rho$.  For
each $A \in V$, we will have $\val_{\cG'}(A)=\val_{\cG}(A)$.
In addition, for each $A \in V$,
every component of $\val_{\cG'}(A)$ will have a root in $\cG'_A$.

We put the variables of $V$ into increasing order of height, and consider
them in that order.  The \slp $\cG'$ will be the last of a sequence of
\slps $\cG_0=\cG, \cG_1,\ldots,\cG_n$, where $n=|V|$,
and $\cG_i=(V_i,S,\rho_i)$ will be made from $\cG_{i-1}$ by considering
the $i$-th variable, $A_i$, in the list.  We have $V_i \supset V_{i-1}$ and
form $\rho_i$ as a modification of $\rho_{i-1}$.  The \slp $\cG_i$ might not be
in Chomsky normal form, since, for some variables $A'$,
$\rho_i(A')$ could be a string of any number of variables $\geq 0$, possibly the
empty string, But it will be clear from our construction
that this is the only obstruction to $\cG_i$ being in Chomsky normal form.

We describe the construction of the sequence $\cG_1,\ldots,\cG_n$,
and prove by induction on $i$ that,
for every variable $A'$ of $\cG_{i}$ not in the sequence $A_{i+1},\ldots,A_n$, 
every component of $\val_{\cG_i}(A')$ has a root in $(\cG_i)_{A'}$.
To prove the $i$-th inductive step we need to verify the existence
of appropriate roots both for the variable $A_i$ and for any new
variables that are defined in the construction of $\cG_i$ from $\cG_{i-1}$.
For each such $A_i$ and all new variables, we can record whether
their values lies in $\Sigma_i^*$ for some $i$.

If $A_i$ has height one, then $\rho(A_i)$ has length at most 1, with $A_i$ 
as its root. So no modification is necessary, and $\cG_i=\cG_{i-1}$.
So now suppose that $A_i$ has height greater than one.
For notational convenience we define $A:= A_i$.
The construction of $\cG_i$ from $\cG_{i-1}$
will involve the definition of some new variables,
and of the images of these and of $A$ under $\rho_i$;
the images under $\rho_i$ and $\rho_{i-1}$ of all other variables will match.

Since $\cG$ is in Chomsky normal form, $\rho_{i-1}(A)=\rho(A)$, and  
$\Ht(A)>1$, we have variables $B,C \in V$ with $\rho_{i-1}(A)=BC$. 
By the inductive hypothesis,
every component of $\val_{\cG}(A)$ that lies entirely within $\val_{\cG}(B)$ or
$\val_{\cG}(C)$ has a root in $(\cG_{i-1})_B$ or $(\cG_{i-1})_C$ (respectively),
and hence in $\cG_{i-1}$.
So no modification is necessary unless  $\val_{\cG}(A)$ contains a component
$u=u_1u_2$ for which $\val_{\cG}(B)=v_1u_1$ and $\val_{\cG}(C)=u_2v_2$ with
$u_1,u_2 \neq \emptyword$.  We suppose that $u$ is such a component. We note
that (since, as discussed earlier when we defined components of words,
any two components of $\val_{\cG}(A)$ are disjoint) any other
component of $\val_{\cG}(A)$ is a subword of either $v_1$ or $v_2$, and hence
of $\val_{\cG}(B)$ or $\val_{\cG}(C)$. 
By the induction hypothesis, $u_1,u_2$ have roots $D_1,D_2$
in $(\cG_{i-1})_B$ and $(\cG_{i-1})_C$ respectively.
By finding the rightmost variable in $\rho^k(B)$ for $k=1,2,\ldots$ and
using our records of which variables have values in $\Sigma_i^*$ for some $i$,
we can locate the variable $D_1$ in polynomial time, and similarly $D_2$.
(In fact we do that in any case to establish whether we are in the situation
where modification is required.)

\begin{figure}
\begin{picture}(300,200)(25,-30)
\put(100,200){\oval(16,12)}\put(96,197){\small $A$}
\put(95,195){\line(-4,-3){38}}
\put(105,195){\line(4,-3){38}}
\put(50,160){\oval(36,12)}\put(36,157){\small $B\!=\!B_0$}
\put(45,153){\line(-1,-2){13}}
\put(50,153){\line(0,-1){26}}
\put(55,153){\line(1,-2){13}}
\put(30,120){$\underbrace{\cdot\cdots\cdots}$}
\put(40,105){$\beta_0$}
\put(70,120){\oval(16,12)}\put(64,117){\small $B_1$}
\put(65,115){\line(-1,-1){32}}
\put(65,114){\line(-1,-2){13}}
\put(67,113){\line(0,-1){28}}
\put(30,80){$\underbrace{\cdot\cdots\cdots}$}
\put(40,65){$\beta_1$}
\put(70,80){\oval(16,12)}\put(64,77){\small $B_2$}
%%%%%%%
\put(70,62){$\vdots$}
\put(70,50){$\vdots$}
%%%%%%%
\put(75,40){\oval(24,12)}\put(64,37){\small $B_{r-1}$}
\put(65,34){\line(-1,-1){30}}
\put(67,34){\line(-1,-2){14}}
\put(70,34){\line(0,-1){28}}
\put(25,0){$\underbrace{\cdots\cdots}$}
\put(35,-15){$\beta_{r-1}$}
\put(75,0){\oval(36,12)}\put(58,-3){\small $B_r\!=\!D_1$}
\put(73,-6){\line(-1,-1){10}}
\put(75,-6){\line(0,-1){10}}
\put(77,-6){\line(1,-1){10}}
%%%%%%%%%%%%%%%
\put(150,160){\oval(36,12)}\put(135,157){\small $C_0\!=\!C$}
\put(145,153){\line(-1,-2){13}}
\put(150,153){\line(0,-1){28}}
\put(155,153){\line(1,-2){13}}
\put(140,120){$\underbrace{\cdot\cdots\cdots}$}
\put(150,105){$\gamma_0$}
\put(130,120){\oval(16,12)}\put(124,117){\small $C_1$}
%%%%%
\put(135,114){\line(1,-1){28}}
\put(133,114){\line(1,-2){14}}
\put(130,114){\line(0,-1){28}}
\put(140,80){$\underbrace{\cdot\cdots\cdots}$}
\put(150,65){$\gamma_1$}
\put(130,80){\oval(16,12)}\put(124,77){\small $C_2$}
%%%%%%%
\put(130,62){$\vdots$}
\put(130,50){$\vdots$}
\put(130,38){$\vdots$}
%%%%%%%
\put(135,30){\oval(24,12)}\put(124,27){\small $C_{s-1}$}
\put(135,24){\line(1,-1){30}}
\put(133,24){\line(1,-2){14}}
\put(130,24){\line(0,-1){26}}
\put(148,-12){$\underbrace{\cdots\cdots}$}
\put(158,-25){$\gamma_{s-1}$}
\put(128,-10){\oval(36,12)}\put(110,-13){\small $D_2\!=\!C_s$}
\put(126,-16){\line(-1,-1){10}}
\put(128,-16){\line(0,-1){10}}
\put(130,-16){\line(1,-1){10}}
%%%%%%%%%%%%%%%%%%
%%%%%%%%%%%%%%%%%%
\put(300,200){\oval(16,12)}\put(296,197){\small $A$}
\put(295,194){\line(-4,-3){38}}
\put(300,194){\line(0,-1){27}}
\put(305,194){\line(4,-3){38}}
\put(300,160){\oval(16,12)}\put(296,157){\small $D$}
\put(298,153){\line(-6,-79){12}}
\put(302,153){\line(6,-79){12}}
\put(250,160){\oval(38,12)}\put(234,157){\small $B'\!=\!B'_0$}
\put(245,153){\line(-1,-2){14}}
\put(250,153){\line(0,-1){28}}
\put(255,153){\line(1,-2){14}}
\put(230,120){$\underbrace{\cdot\cdots\cdots}$}
\put(240,105){$\beta_0$}
\put(270,120){\oval(16,12)}\put(264,117){\small $B'_1$}
\put(265,113){\line(-1,-1){32}}
\put(267,113){\line(-1,-2){14}}
\put(270,113){\line(0,-1){28}}
\put(230,80){$\underbrace{\cdot\cdots\cdots}$}
\put(240,65){$\beta_1$}
\put(270,80){\oval(18,12)}\put(264,77){\small $B'_2$}
%%%%%%%
\put(270,62){$\vdots$}
\put(270,50){$\vdots$}
%%%%%%%
\put(272,40){\oval(26,12)}\put(261,37){\small $B'_{r-1}$}
\put(263,34){\line(-1,-1){30}}
\put(267,34){\line(-1,-2){14}}
\put(224,-1){$\underbrace{\cdots\cdots}$}
\put(234,-16){$\beta_{r-1}$}
%%%%%%%%%%%%%%%
\put(350,160){\oval(38,12)}\put(335,157){\small $C'_0\!=\!C'$}
\put(345,153){\line(-1,-2){14}}
\put(350,153){\line(0,-1){28}}
\put(355,153){\line(1,-2){14}}
\put(340,120){$\underbrace{\cdot\cdots\cdots}$}
\put(350,105){$\gamma_0$}
\put(330,120){\oval(16,12)}\put(324,117){\small $C'_1$}
%%%%%
\put(335,113){\line(1,-1){28}}
\put(333,113){\line(1,-2){14}}
\put(330,113){\line(0,-1){28}}
\put(340,80){$\underbrace{\cdot\cdots\cdots}$}
\put(350,65){$\gamma_1$}
\put(330,80){\oval(18,12)}\put(324,77){\small $C'_2$}
%%%%%%%
\put(330,62){$\vdots$}
\put(330,50){$\vdots$}
\put(330,38){$\vdots$}
%%%%%%%
\put(335,30){\oval(26,12)}\put(324,28){\small $C'_{s-1}$}
\put(335,24){\line(1,-1){30}}
\put(333,24){\line(1,-2){14}}
\put(348,-10){$\underbrace{\cdots\cdots}$}
\put(358,-25){$\gamma_{s-1}$}
\put(286,-12){\oval(16,12)}\put(280,-15){\small $D_1$}
\put(284,-18){\line(-1,-1){10}}
\put(286,-18){\line(0,-1){10}}
\put(288,-18){\line(1,-1){10}}
	\put(313,-12){\oval(16,12)}\put(306,-15){\small $D_2$}
\put(311,-18){\line(-1,-1){10}}
\put(313,-18){\line(0,-1){10}}
\put(315,-18){\line(1,-1){10}}
\end{picture}
\caption{Adjusting the \slp below $A$}
\label{fig:adjustment}
\end{figure}
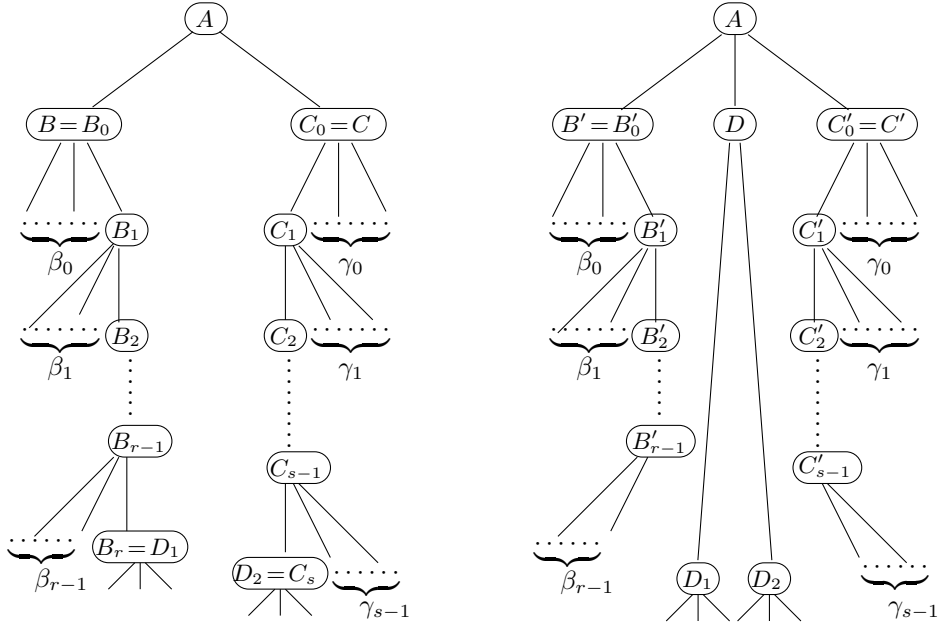

Now we introduce a new variable $D$, and define $\rho_i(D):=D_1D_2$.
We also introduce new variables $B',C'$ and set $\rho_i(A)=B'DC'$. 
We want to define $\rho_i(B')$ and $\rho_i(C')$ to ensure that 
$\val_{\cG_i}(B')=v_1$, and $\val_{\cG_i}(C')=v_2$,
and this motivates the definitions that now follow.

First we note the existence of a sequence of variables
$B_0=B,B_1,\ldots, B_r=D_1$ in $V_{i-1}$,
for which each $B_{j+1}$ is the last letter of $\rho_{i-1}(B_j)$;
that is $\rho_{i-1}(B_j)=\beta_j B_{j+1}$, for some
$\beta_j \in V_{i-1}^*$.

Similarly, we find a sequence of variables
$C_0=C,C_1,\ldots, C_s=D_2$ in $V_{i-1}$,
such that, for each $j$, $\rho_{i-1}(C_j)=C_{j+1}\gamma_j$, for some
$\gamma_j \in V_{i-1}^*$.
We define $B'_0:=B'$, $C'_0:=C'$ and introduce new variables $B'_j$
for each $j=1,\ldots,r-1$, and $C_j'$ for each $j=1,\ldots,s-1$.
We define 
\[
V_i := V_{i-1} \cup \{B'_0,B'_1,\ldots B'_{r-1},C_0',C'_1,\ldots C'_{s-1},D\}.\]
Then we define 
\[ \rho_i(B'_0):=\beta_0B'_1,\quad
\rho_i(B'_j):=\beta_jB'_{j+1}\ 
\hbox{\rm for}\,j=1,\ldots,r-2,\,\quad
\rho_i(B'_{r-1}):=\beta_{r-1},\]
and similarly
\[ \rho_i(C'_0):=\gamma_0C'_1,\quad
\rho_i(C'_j):=\gamma_jC'_{j+1}\ 
\hbox{\rm for}\,j=1,\ldots,s-2,\,\quad
\rho_i(C'_{s-1}):=\gamma_{s-1}.\]
and $\rho_i(E):= \rho_{i-1}(E)$ for all variables
$E \in V_{i-1}\setminus \{ A \}$.
Figure~\ref{fig:adjustment} illustrates this adjustment of
$\cG_{i-1}$ to $\cG_i$.

This completes our construction of $\cG_i$, which is certainly acyclic.
Our definition of new variables and of their images under $\rho_i$
was designed to ensure that $\val_{\cG_i}(A_i)=\val_{\cG}(A_i)$.
We note that $\Ht_{\cG_i}(B') = \Ht_{\cG'_{i-1}}(B)$,
$\Ht_{\cG_i}(C') = \Ht_{\cG_{i-1}}(C)$, and
$\Ht_{\cG_i}(D) \le \Ht_{\cG_{i-1}}(A_i)$ (equality can occur here if
$D_1=B$ or $D_2=C$).
So $\Ht_{\cG_i}(A_i) \le \Ht_{\cG_{i-1}}(A_i)+1$, for each $i$, and hence
$\Ht_{\cG_i}(S) \le \Ht_{\cG_{i-1}}(S)+1$.

We need to verify our claim that, for all variables $A' \in \cG_i$ that are not
in the sequence $A_{i+1},\ldots,A_n$, every component of $\val_{\cG_i}(A')$
has a root in $(\cG_i)_{A'}$. 
This is true for all $A' \in V_{i-1} \setminus \{A_i,A_{i+1},\ldots,A_n\}$,
because $\rho_i(A')=\rho_{i-1}(A')$ and hence
 $\val{_{\cG_i}}(A')=\val_{\cG}(A')$ for all such $A'$.
Our construction of $\cG_i$ was designed to ensure that our claim is true for
$A'=A_i$, and it is clearly true for $A'=D,D_1$ and $D_2$.
Finally, for the new variables $B_j'$ and $C_j'$, the components of
$\val_{\cG_i}(B_j')$ and $\val_{\cG_i}(C_j')$ are components of
$\val_{\cG_{i-1}}(B_j)$ and $\val_{\cG_{i-1}}(C_j)$ respectively and,
by the inductive hypothesis, they have roots in
$(\cG_{i-1})_{B_j}$ and $(\cG_{i-1})_{C_j}$.
It follows immediately from the definitions of $\rho_i(B_j')$ and
$\rho_i(C_j')$ that these components have roots in
$(\cG_{i})_{B_j'}$ and $(\cG_{i})_{C_j'}$.

After the process is complete, we have $\Ht_{\cG'}(S) \le n+\Ht_{\cG}(S)
\le 2n$. So, during each of the $n$ steps of this process the number of
variables we have added has been at most than $4n$, and hence the process is
bounded by a polynomial in $|V|$, so certainly by a polynomial in $|\cG|$.
We complete the proof by putting the final \slp into Chomsky normal
form. Since this involves only the addition of new variables, it does
not affect the property that all components of $w$ have roots.
\end{proof}

\begin{definition}\label{def:split}
Let $w$ be a word over $\Sigma$ and suppose that $u:=w[i:j)$ is a
subword of $w$.  
We say that $u$ {\em splits a component} of $w$ if $u$ starts or ends
part way through a component of $w$, but is not a subword of a single component;
otherwise we say that  $u$ {\em splits no components} of $w$.
\end{definition}

So if $u$ splits no component of $w$, then either $u$ is a proper subword of a
component of $w$, or $u$ is a concatenation of components of $w$
and of subwords in $(\Sigma \setminus \cH)^*$. 
In that second case, there exist integers $k,l$ with $k<l$ such that
$\hat{u} = \hat{w}[k:l)$; we shall sometimes choose to write
$w[[k:l))$ rather than $w[i:j)$ as a notation for this subword $u$ of $w$.

Now let $\cG=(V,S,\rho)$ be an \slp, \cslp, \tslp or \tcslp for the group $G$,
and let $w:= \val(\cG)$. 
We say that $\cG$ {\em splits no components} (or is {\em non-splitting}) if 
for any variable $A$ of $\cG$, whenever $\val(A)$ occurs as a subword of $w$
with root $A$, then that that subword splits no components of $w$. 

Now suppose that $A,B$ are variables in a \cslp $\cG$ and that
$\rho(A)=B[i:j)$ (so that $\val_\cG(A)=\val_\cG(B)[i:j)$). We say that 
the cut operator $B[i:j)$ {\em splits a component} if
the subword $w_B[i:j) = w_A$ splits a component of $w$, or if
it is a proper subword of a component of $w$;
otherwise $B[i:j)$ is called {\em non-splitting}.
Notice that for cut operators to be non-splitting,
we are also excluding the possibility of
$w_A$ being a proper subword of a component of $w_B$.
If $B[i:j)$ is a non-splitting cut-operator 
and $k,l$ are the integers defined by
$\val(B)[i:j) = \val(B)[[k:l))$, it is often 
convenient to specify the cut operator in terms of $k$ and 
$l$ rather than $i$ and $j$,  that is, as $B[[k:l))$.
In this case,
we say that the cut operator is {\em specified relative to compression}.

In general
the \slps (and \cslps, \tslps and \tcslps) that we shall construct 
in this article, as well as cut operators within the
\tcslps, will not split components,
and the cut operators 
will be specified relative to compression.

\begin{remark}\label{rem:subword}
Suppose that $\cG$ is an \slp as in the conclusion of
Proposition~\ref{prop:comproot} (where it is called $\cG'$); that is,
for each variable $A$ of $\cG$, every component of $\val(A)$
has a root in $\cG_A$. Then $\cG$ splits no components.
Moreover, if $\cG$ is trimmed and $\val(\cG)$ is $\nf$-reduced, then so is
$\val(A)$, for every variable $A$ of $\cG$.
\end{remark}
\begin{proof}
Suppose that the subword $u = \val(A)$ of $w := \val(\cG)$ splits a component
of $w$; so we have $u = u_1u_2$ with $u_1$ and $u_2$ nonempty, where either
$u_1$ is a proper suffix or $u_2$ is a proper prefix of a component $v$ of $w$.
But, by assumption, $v$ has a root $B$ in $\cG$, and these occurrences
of $\val(A)$ and $\val(B)$ in $w$ overlap without one being a subword
of the other, which is not possible. The final assertion now follows from
Remark~\ref{rem:subclosure}.
\end{proof}

In order to build \cslps that do not split components, we shall need the
following result.

\begin{corollary} \label{cor:gamlen}
Let $\cG$ be an \slp over our selected generating set $\Sigma$ for $G$, 
let $w := \val(\cG)$, and suppose that every component of $w$ has a root in
$\cG$ (and hence $\cG$ splits no components). 
Then given $k,l$ with $0\leq k<l\leq |\hat{w}|$, 
in time that is polynomial in
$|\cG|$, we can compute the integers $i,j$ with $w[i:j) = w[[k:l))$.
Conversely, given $i$ and $j$ such that the subword $w[i:j)$ of $w$ is a union
of complete components of $w$ and subwords in $(\Sigma \setminus \cH)^*$,
we can compute $k$ and $l$ with $w[[k:l))=w[i:j)$, in polynomial time.
\end{corollary}
\begin{proof}
By regarding the roots of components of $w$ as new terminals, we
can regard $\cG$ as an \slp $\widehat{\cG}$ over some finite subset of the
infinite alphabet $\widehat{\Sigma}$ with  $\val({\widehat{\cG}})=\hat{w}$.
Our lists of variables $A$ of $\cG$ for which $\val(A) \in \Sigma_i^*$ for
some $i$ enable us to identify such variables.
We can then use Proposition \ref{prop:slp_results}\,(iv) to
compute \slps over $\widehat{\Sigma}$ with values
$\hat{w}[:l)$ and $\hat{w}[k:l)$. Then by reinterpreting them as
\slps over $\Sigma$, we can (by Proposition \ref{prop:slp_results}\,(iii))
compute their lengths and thereby compute $i$ and $j$.

For the converse, we compute \slps for $w[:j)$ and $w[i:j)$, regard
then as \slps with values $\hat{w}[:l)$ and $\hat{w}[k:l)$
over a finite subset of  $\widehat{\Sigma}$, and then compute their lengths.
\end{proof}

\begin{lemma} \label{lem:ftwd}
Let $\cG$ be an \slp over our selected generating set $\Sigma$ for $G$, 
and let $w := \val(\cG)$.
Suppose that the components of $w$ are all shortlex reduced and that
$\hat{w}$ is a $(\lambda,c)$-quasigeodesic that does not backtrack for
some constants $\lambda \ge 1$ and $c \ge 0$. Then, in
time polynomial in $|\cG|$, we can construct an \slp with value $\nf(w)$.
\end{lemma}
\begin{proof}
We assume that $\cG$ is trimmed and in Chomsky normal form and that
all of its components have roots. So, in particular, for all variables $A$ of
$\cG$, $\val_\cG(A)$ arises (possibly multiple times) as a subword of $w$.
The hypotheses allow us to apply Proposition \ref{prop:ftqd} with $v = w$,
$u = \nf(w)$, and $w_1=w_2=1$. Let $p$ and $q$  be the paths in $\Gamma$ that 
are labelled by
$\nf(w)$ and $w$, and $\hat{p}$ and $\hat{q}$ the corresponding paths in
$\widehat{\Gamma}$. Then $p$ and $q$ $\ee'$-fellow travel, where $\ee'$ is a constant that
depends only on $G$, $\Sigma$, $\lambda$ and $c$, and 
all vertices on $\hat{q}$ have at least one corresponding vertex on
$\hat{p}$.  As in the proof of  Proposition \ref{prop:ftqd}, we define
$\ee_1 := \ee(\lambda,c,0)$ and $\ee_2 := \ee(\lambda,c,\ee_1)$,
where $\ee()$ is as defined in Proposition~\ref{prop:bcpp} 

Now consider some instance of $\val(A)$ as a subword $w_1$ of $w$ that
is derived from $A$ and labels a subpath $q_1$ of $q$. As we observed in
Remark~\ref{rem:subword}, either (this instance of) $w_1$ is a concatenation
of complete components of $w$ and subwords in $(\Sigma \setminus \cH)^*$,
or else $w_1 \in \Sigma_i^*$ for some $i$ and $w_1$ is a proper subword of a
component of $w$.

In the first of these cases (Case 1), the start and end vertices of subpath $q_1$ of $q$ are also vertices of
$\hat{q}$, and the fact that vertices on
$\hat{q}$ have corresponding vertices on $\hat{p}$ implies that there
exist $\alpha,\beta \in \Sigma^*$ with $|\alpha|,|\beta| \le \ee'$ such that
$\nf(\alpha\val_\cG(A)\beta^{-1})$ is a subword of $\nf(w)$ that labels
a subpath $p_1$ of $p$ whose start and end vertices are also vertices of $\hat{p}$.  (Note that we
are using Remark~\ref{rem:subclosure} here, which ensures that the subword of
$\nf(w)$ labelled by this subpath of $p$ is $\nf$-reduced.)

In the second case we see from the proof of Proposition \ref{prop:ftqd} that,
if (Case 2a) the component of $w$ of which $w_1$ is a proper subword has length
greater than $\ee_2$, then there exist $\alpha,\beta \in \Sigma_i^*$ with
$|\alpha|,|\beta| \le \ee'$ such that $\nf(\alpha\val_\cG(A)\beta^{-1})$
labels a subword of a connected component of $\nf(w)$.  (Here we using the
fact that the connected components are shortlex reduced words and so their
subwords are also shortlex reduced and hence $\nf$-reduced.)

In Cases 1 and 2a, we say that this instance of the subword $w_1=\val_\cG(A)$
of $w$ {\em corresponds to} the subpath of $p$ that is labelled by
$\nf(\alpha\val_\cG(A)\beta^{-1})$.
If $w_1$ is a proper subword of a component of $w$ of length at most $\ee_2$
(Case 2b), then we do not attempt to define a corresponding subpath of $p$. 

We shall construct an \slp $\cS$ with value $\nf(w)$. As usual, we do this
by processing the variables of $\cG$ in order of increasing height.
For each variable $A$ of $\cG$ and each pair of words
$\alpha,\beta \in \Sigma^*$ with $|\alpha|,|\beta| \le \ee'$, we attempt
to define a variable $A_{\alpha,\beta}$ of $\cS$ with
$\val_{\cS}(A_{\alpha,\beta}) = \nf(\alpha\val_\cG(A)\beta^{-1})$.
We shall not necessarily succeed in doing this for every such $\alpha,\beta$,
but we shall do so at least for each $\alpha,\beta$ for which
$\val(A)$ corresponds to $\nf(\alpha\val_\cG(A)\beta^{-1})$ as described above.
After carrying out this process for all variables $A$ of $\cG$, we complete the
proof by letting the start variable of $\cS$ be $S_{\emptyword,\emptyword}$
for $S = S_\cG$.

Suppose first that either $A$ has height $0$, or that $\val(A) \in \Sigma_i^*$
for some $i$ and $|\val(A)| \le \ee_2$.  Then we compute
$\nf(\alpha \val(A) \beta^{-1})$ for all $\alpha,\beta \in \Sigma^*$
with $|\alpha| ,|\beta|\leq \ee'$.
Since $|\alpha \val(A) \beta|$ is bounded above by the constant $\ee_2+2\ee'$,
we can do this in time bounded by a constant, by
using the word acceptor in the asynchronous automatic structure of $G$ to
enumerate the words in the accepted language in order of increasing length until
one is found that is equal in $G$ to $\alpha \val(A) \beta^{-1}$,
and then define $A_{\alpha,\beta}$ with
$\rho_{\cS}(A_{\alpha,\beta}) = \nf(\alpha \val(A) \beta^{-1})$.

Otherwise, we have $\rho_\cG(A) = BC$ for variables $B$ and $C$ of $\cG$ that
we have processed already. We proceed as follows for each pair
$\alpha,\beta \in \Sigma^*$ with $|\alpha|,|\beta| \le \ee'$.
For all words $\gamma \in \Sigma^*$ with $|\gamma| \le \ee'$ for which we have
defined variables $B_{\alpha,\gamma}$ and $C_{\gamma,\beta}$ for $\cS$,
we check whether the word
$\val_{\cS}(B_{\alpha,\gamma})\val_{\cS}(C_{\gamma,\beta})$
lies in  the language of the word acceptor, which we can do in polynomial time
by Proposition~\ref{prop:slp_results}\,(v).
If so, then we introduce the new variable $A_{\alpha,\beta}$ for $\cS$, define
$\rho_{\cS}(A_{\alpha,\beta}) = B_{\alpha,\gamma} C_{\gamma,\beta}$, and
move on to the next pair of words $\alpha,\beta$.

\begin{figure}
\begin{picture}(340,85)(0,-20)
\put(0,50){\circle*{4}}
\put(75,0){\circle*{4}}
\put(145,50){\circle*{4}}
\put(135,0){\circle*{4}}
\put(170,50){\circle*{4}}
\put(170,0){\circle*{4}}
\put(205,50){\circle*{4}}
\put(215,0){\circle*{4}}
\put(340,50){\circle*{4}}
\put(265,0){\circle*{4}}

\put(0,50){\line(1,0){340}}
\put(0,50){\line(3,-2){75}}
\put(75,0){\line(1,0){190}}
\put(265,0){\line(3,2){75}}

\put(154,0){\vector(1,0){0}}
\put(194,0){\vector(1,0){0}}
\put(159,50){\vector(1,0){0}}
\put(189,50){\vector(1,0){0}}
\put(147,5){$w_2$}
\put(186,5){$w_3$}

\put(4,-8){\vector(1,0){166}}
\put(170,-8){\line(1,0){166}}
\put(165,-16){$w$}
\put(4,58){\vector(1,0){166}}
\put(170,58){\line(1,0){166}}
\put(155,63){$\nf(w)$}

\put(145,50){\vector(-1,-5){6}}
\put(139,20){\line(-1,-5){4}}
\put(170,50){\vector(0,-1){28}}
\put(170,22){\line(0,-1){22}}
\put(205,50){\vector(1,-5){6}}
\put(211,20){\line(1,-5){4}}
\put(150,42){$u_2$}
\put(182,42){$u_3$}

\put(128,23){$\alpha$}
\put(173,23){$\gamma$}
\put(213,23){$\beta$}

\end{picture}
\caption{Subwords of fellow travelling words $w$ and $\nf(w)$}
\label{fig:ftsubwd}
\end{figure}
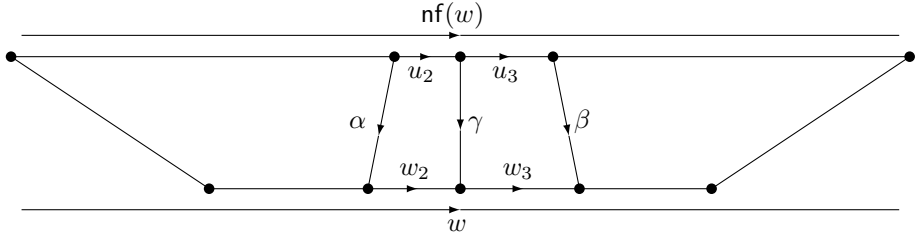

We need to show that we succeed in defining $A_{\alpha,\beta}$ whenever
$A$ corresponds to %$\nf(\alpha\val_\cG(A)\beta)$ 
$\nf(\alpha\val_\cG(A)\beta^{-1})$ as described above.
Let $w_1 := \val_\cG(A)$, $w_2 := \val_\cG(B)$ and $w_3 := \val_\cG(C)$,
so $w_1=w_2w_3$.  Since we are assuming that $|\val_\cG(A)| > \ee_2$,
all instances of $w_1$ as subwords of $w$ must lie in Case 1 or Case 2a,
and (since all components of $\cG$ have roots)
we see then that the same applies to the instances of $w_2$ and $w_3$
that arise from the derivation $\rho(A) = BC$.  So the vertex of $\Gamma$
corresponding to the end of the subword $\val_\cG(B)$ is at
distance at most $\ee'$ from a vertex in the corresponding subword
$u_1 := \nf(\alpha\val_\cG(A)\beta^{-1})$ of $\nf(w)$ such that
the associated subwords $u_2$ and $u_3$ of $u_1$ correspond to $w_2$
and $w_3$, respectively.
See Figure \ref{fig:ftsubwd}.

In other words,
there exists $\gamma \in \Sigma^*$ with $|\gamma| \le \ee'$ such
that $\val_\cG(B)$ corresponds to $\nf(\alpha \val_\cG(B) \gamma^{-1})$ and
$\val_\cG(C)$ corresponds to $\nf(\gamma \val_\cG(C) \beta^{-1})$, in which
case the variables $B_{\alpha,\gamma}$ and $C_{\gamma,\beta}$ have been
defined, and so we successfully define $A_{\alpha,\beta}$
with $\rho_{\cS}(A_{\alpha,\beta}) = B_{\alpha,\gamma} C_{\gamma,\beta}$.
\end{proof}

Note that the \slp $\cS$ that is constructed in the above lemma can have size
up to $(\ee')^2$ times the size of $\cG$. The following result,
guarantees a size that is bounded as a function of the length
of the derived word $\hat{w}$.

\begin{proposition}\label{prop:slpshort}  
Let $\cG$ be an \slp over our selected generating set $\Sigma$ for $G$, and
let $w:=\val(\cG)$.  Then we can construct an \slp $\cS$ with value equal to
$\nf(w)$ and with size at most
$\max(\CC|\hat{w}|\log(|w|),1)$ in time that
is bounded by $\jj(|\hat{w}|,|\cG|)$, for some constant $\CC$ and some
(increasing) polynomial function $\jj()$.
\end{proposition}
\begin{proof}
The proof is by induction on $|\hat{w}|$.
We shall not attempt to specify the constant $\CC$ in the statement precisely,
but it would be possible to do so by following the calculation at the end of the
proof. This constant will depend on $|\Sigma|$
%to the parabolic subgroups $H_i$ of $G$
and on the constant $\DD$ from Proposition~\ref{prop:nfwdlenbd}.
Similarly we shall not attempt to specify
the polynomial $\jj()$ in the statement, but merely observe that all steps
in the proof can be done in time polynomial in $|\hat{w}|$ and
$\log(|w|)$, and that there are $|\hat{w}|$ steps in the inductive proof.
By Proposition \ref{prop:comproot} we may assume that all components of
$\cG$ have roots.  

There is nothing to prove if $|\hat{w}|=0$. Otherwise,
we can write $\hat{w} = \hat{v}\hat{z}$ with $|\hat{z}|=1$,
where also $w = vz$, and so $w=_G vz$,
and we can assume that we have found an \slp $\cS_v$ with
$\val(\cS_v)=\nf(v)$, where the size of  $\cS_v$ is bounded as required.
We denote by $\vnf$ the word $\nf(v) = \val(\cS_v)$ and by $\wnf$ the
word $\nf(w)$. Then $\wnf =_G \vnf z$.

Suppose first that $z$ represents an element of $H_i$ for some component
subgroup $H_i$ (so $z \in \Sigma_i^*$) and that $\vnf$ ends in  
a letter from $\Sigma_i$.
Then let $y$ be the $H_i$-component of $\vnf$ that contains this letter;
so we have $\vnf=uy$ and $\widehat{\vnf} = \hat{u}\hat{y}$,
and it follows from the description of $\cL$ in Proposition \ref{prop:automatic}
that $\nf(u) = u$.
The word $u$ might not be a prefix of $w$, but
$w=_G u(yz)$, and $\hat{u}$ does not end in a letter from $\Sigma_i$.
In this situation, we replace $\vnf$ by $u$, derive $\cS_u$
with value $u=\nf(u)$ from $\cS_v$, and replace $z$ by $yz$. 

So in any case, we may now assume that, if $z \in\Sigma_i^*$ for some $i$,
then $\vnf$ does not end in a letter from $\Sigma_i$.
Let $\znf := \nf(z)$. So $\znf = \slex(z)$ when $z \in\Sigma_i^*$; otherwise
$z$ consists of a generator that does not lie in any parabolic subgroup and,
since we are assuming that the elements of $\Sigma$ represent
distinct elements of $G$, we have $|z|=1$ and $\znf = z$.
In the former case we can compute
an \slp $\cS_z$ with $\val(\cS_z) = \znf$  and $|\cS_z| \le
\max(4|\Sigma_i|\log_2(\znf)),1)$ by Proposition \ref{prop:abslp}.
So we can in any case compute an \slp $\cG$ with value $\vnf \znf$.

Now the path $\hat{p}$ in $\widehat{\Gamma}$ labelled by
$\widehat{\vnf \znf} = \widehat{\vnf}\widehat{\znf}$,
as an extension of a geodesic by an element of
$\widehat{\Sigma}$, is a $(1,2)$--quasigeodesic.
We claim that $\hat{p}$ does not backtrack. For suppose that it does.
Then the final edge of $\hat{p}$ must penetrate the same $H_{i}$-coset
as one of the other edges of $\hat{p}$, for some index $i$,
and then we must have $z \in H_i$. So some suffix of
$\widehat{\vnf}$ represents an element of $H_i$ and, since the word
$\widehat{\vnf}$ is geodesic, this suffix must have length $1$. But then
$\widehat{\vnf}$ ends in a letter in $\Sigma_i$, contrary to assumption.

We can now apply Lemma \ref{lem:ftwd} to the \slp $\cG$ with value $\vnf \znf$
and compute an \slp $\cS_0$ with value $\wnf:=\nf(\vnf z)$.
By Proposition~\ref{prop:comproot}, we can modify $\cS_0$ as necessary (in
polynomial time), to ensure that all of its components have roots.

We now explain how to modify $\cS_0$ such that its size is bounded
by $\CC|\hat{w}|\log(|w|)$ for some constant $\CC$. 
Let $\wnf[i_r:j_r)$ for  $1 \le r \le t$ be the components of $\wnf$,
and let $A_r$ be the root of $\wnf[i_r:j_r)$ in $\cS_0$.
Then by Lemma~\ref{lem:ab_polytime} we can
construct an \slp $\cS_r$ with $\val(\cS_r) = \val(A_r)$ and
$|\cS_r| \le \max(4|\Sigma|\log_2(|\val(\cS_r)|)|,1)$, 
and define $B_r := S_{\cS_r}$.
Note that $|\cS_r|$ is bounded by a multiple of $\log(|\wnf|)$,
which is itself  bounded by a multiple of $\log(|w|)$ by
Proposition~\ref{prop:nfwdlenbd}.
We now define $\cS$, whose variables consist of $S_\cS$ together with all of
the variables of each $\cS_r$ with $r \geq 1$.
For each $A \in \cS_r$, we define $\rho_\cS(A) := \rho_{\cS_r}(A)$,
and further
\[ \rho(S_{\cS}) = \wnf[j_0:i_1)B_1\wnf[j_1:i_2)B_2 \cdots \wnf[j_{t-1}:i_t)
  B_t\wnf[j_t,i_{t+1}), \]
where  $j_0:=0$ and $i_{t+1} := |\wnf|$.
Then  the \slp $\cS$ satisfies the required bound on its size.
\end{proof}

%%%%%%%%%%%%%%%%%%%%%%%%%%%%%%%%%%%%%%%%%%%%%%%%%%%%%%%%%%%%
\section{Converting \tcslps and \tslps to \slps}
\label{sec:convert}

The proof of our main result falls into two parts, the first part constructing,
from the input \slp $\cG$, a \tcslp that defines the $\nf$-reduced
representative $\nf(\val(\cG))$ of its value,
and then the second part converting that \tcslp into a \slp with the same value;
this is the same strategy as was applied to prove the corresponding result
\cite[Theorem 6.7]{HLS} for hyperbolic groups, and our proofs of the component
results are based on the proofs in \cite{HLS}.
%%%%%%%%%%%%%%%%%%%%%%%%%%%%

The construction of a \tcslp accepting $\nf(\val(\cG))$ is described 
in the final section of the paper, Section~\ref{sec:slextcslp}.
This section is devoted to the proof of the following conversion theorem:

\begin{theorem} \label{thm:convert}
Let $G$ be a group hyperbolic relative to a collection
of free abelian subgroups, and suppose that a generating set $\Sigma$ for $G$,
and integer $L$ are selected as in
Sections ~\ref{subsec:relhyp_param1}, \ref{subsec:relhyp_param2}.
Let $\cT$ be an $\nf$-reduced non-splitting \tcslp for $G$ over $\Sigma$,
with $\JJ_\cT \le L$.
Suppose further that each cut operator of $\cT$ is non-splitting
and specified relative to compression.
Then we can construct, in time polynomial in $|\cT|$, an $\nf$-reduced \slp
$\cS$ over $\Sigma$ with $\val(\cS) = \val(\cT)$, whose size is bounded by a 
polynomial in $|\cT|$.
\end{theorem}
We choose to stress the polynomial bound on the size of $\cS$, although it
follows from the polynomial bound on time.

The proof of Theorem~\ref{thm:convert} is split into the two results that
follow, Propositions~\ref{prop:tslp-slp} and~\ref{prop:tcslp-tslp}.
The first of these computes from a given
$\nf$-reduced \tslp $\cU$ an \slp $\cS$ with the same value, and the second
computes an $\nf$-reduced \tslp $\cU$ from an $\nf$-reduced \tcslp $\cT$.
This follows the strategy of the proof of the corresponding
result in \cite{HLS}, and our proofs adapt those of
the components of that proof.

There are complications in our proofs that do not arise in \cite{HLS}, 
resulting  partly from the fact that we are
using Proposition \ref{prop:backup2} rather than \cite[Lemma 4.4]{HLS}, and
partly because many of the upper bounds on lengths of words are bounds on
their lengths as words over $\widehat{\Sigma}$ rather than over $\Sigma$.

\begin{proposition}\label{prop:tslp-slp}
Let $G,\Sigma$ and $L$ be as in Theorem~\ref{thm:convert}. Then,
given an $\nf$-reduced non-splitting \tslp $\cU$ for $G$ over $\Sigma$ with
$\JJ_\cU \le L$,
we can construct, in time polynomial in $|\cU|$, an \slp
$\cS$ over $\Sigma$ with $\val(\cS) = \val(\cU)$,
whose size is bounded by a polynomial function of $|\cU|$.
\end{proposition}
\begin{proof}
Let $\ff$ be the linear function in the conclusion of
Proposition \ref{prop:backup2}. Modifying $\ff$ as necessary, we can assume
that $\ff$ is an increasing function with $\ff(n) \ge n$ for all $n$.

The result is trivial if $\val(\cU) = \epsilon$, so suppose not.
We defined the height of a variable in $\cU$ in
Section~\ref{subsec:extend_slps}.
We also pointed out that the right hand sides $\rho(A)$
that lie in $(V_{\cU} \cup \Sigma)^*$ may be assumed to have the form
$a \in \Sigma$ or $BC$ with $B,C \in V_{\cU}$, and we shall assume that here.

We define the {\em tether-height} $\tHt(A)$
and {\em tether-depth} $\tD(A)$ of a variable $A$, via
\begin{eqnarray*}
\tHt(A)&:=&\left \{ 
\begin{array}{l} 0 \\ \max\{\tHt(B),\tHt(C)\}\\ \tHt(B)+1 \end{array} \right .
					  \, \hbox{\rm if} \,
	  \begin{array}{l} \rho_\cU(A)\in \Sigma^*\\ \rho_\cU(A)=BC\\
	  \rho_\cU(A) = B \langle \alpha,\beta \rangle, \end{array}  \\
\tD(A) &:=& \tHt(S) - \tHt(A) +1,\end{eqnarray*}
where $S=S_\cU$.
By removing unused variables, we may assume that $S$ has maximal height and
maximal tether-height, so that every variable has positive tether-depth.

In the proof, it is convenient to assume that $A$, $B$ and $C$ have the
same tether-depths in all productions of type $A \to BC$. To achieve this,
we can increase the tether-depth of any variable $A$, if necessary, by
introducing a new redundant variable $X$ together with a production
$X \to A\langle  \emptyword,\emptyword \rangle$. This will not affect the
maximality of the height and tether-height of $S$.

We process the variables of $\cU$ in order of increasing height.
Since the number of variables of $\cU$ is certainly bounded by $|\cU|$,
in order to get the bounds we need on time and space, it will
be sufficient to bound the time spent processing each variable $A$, 
together with the length of right-hand sides added during each such step.

As we process each variable $A$ of $\cU$ we shall
either define a copy of $A$ within $\cS$ or a
set of bounded size of new variables for $\cS$
that are associated with $A$, and for each new variable of $\cS$ that
we introduce we shall define its right-hand side in $\cS$.

Our construction will ensure the following,
for any variable $A$ of $\cU$, where $w:= \val_\cU(A)$.
(Recall that, since $\cU$ is $\nf$-reduced, $w=\nf(w)$.)
\begin{mylist}
\item[(i)]
If $|\hat{w}| \le (8\,\tD(A) + 1)\ff(2L)$, then
$\cS$ contains a copy of $A$, and the \slp $\cS_A$ 
is computed from $\cU_A$ in time bounded by $\jj(|\hat{w}|,|\cU|)$, and 
has size at most $\CC|\hat{w}|\log(|w|)$,
where $\CC$ and $\jj$ are the constant and polynomial
of Proposition ~\ref{prop:slpshort}.
Note that Proposition \ref{prop:tcslplenbd} implies that the size of
$\cS_A$ is bounded by a polynomial function of the input size $|\cU|$ in this
case.
\item[(ii)] If $|\hat{w}| > (8\,\tD(A) + 1)\ff(2L)$, then
$w$ decomposes as a concatenation $\ell_Aw'r_A$ with
$|\widehat{w'}| \ge \ff(2L)$, and  \[4\ff(2L)\tD(A) \le
|\widehat{\ell_A}|,|\widehat{r_A}| \le (4\tD(A) + \Ht(A))\ff(2L).\]
As $A$ is processed, new variables $A_\ell$ and $A_r$
are adjoined to $\cS$
as the roots of \slps with values $\ell_A,r_A$
(with size at most $\CC|\widehat{\ell_A}|\log(|\ell_A|)$ and
$\CC|\widehat{r_A}|\log(|r_A|)$),
as well as variables $A_{\alpha,\beta}$ as the roots of \slps with values
$\nf(\alpha w' \beta^{-1})$,
for each $\alpha,\beta \in  \Sigma^*$ with $|\alpha|,|\beta| \le L$.

Each of the subwords $\ell_A$, $w'$ and $r_A$ of $w$ is a union of complete
components of $w$ and subwords in $(\Sigma \setminus \cH)^*$,
and is in normal form.
\end{mylist}

Suppose that, while processing the variables of $\cU$ in increasing order
of height, we have reached the variable $A$ of $\cU$.
Let $w:= \val_\cU(A)$.

We consider three different possible cases.

{\bf Case 1.} $\rho_{\cU}(A) \in \Sigma$. We define
a variable $A$ within $\cS$, and define $\rho_\cS(A) := a$.

{\bf Case 2.} $\rho_\cU(A)=BC$ for variables $B,C$ of $\cU$.
Recall that we are assuming that
$\tD(A),\tD(B),\tD(C)$ are all equal; let $\td$ be their common value.
Let $u := \val_\cU(B)$, $v:=\val_\cU(C)$, so that $w=uv$.
Since we are assuming that $\cU$ is non-splitting, either $uv \in \Sigma_i^*$
for some $i$, in which case $|\hat{u}|,|\hat{v}|,|\hat{w}|\le 1$,
or $\hat{w} = \hat{u}\hat{v}$.

Suppose first (Case 2.1) that
$|\hat{u}|,|\hat{v}| > (8\td + 1) \ff(2L)$
(so we do not have $uv \in \Sigma_i^*$).
Then, when we processed the variables $B,C$,
we computed \slps with values $\ell_B, r_B, \ell_C, r_C$ such that:
\begin{mylist}
\item[(i)]  $4\td \ff(2L) \le |\widehat{\ell_B}|$, $|\widehat{r_B}| \le
   (4\td + \Ht(B))\ff(2L)$,
\item[(ii)]  $4\td \ff(2L) \le |\widehat{\ell_C}|$, $|\widehat{r_C}| \le
(4\td + \Ht(C))\ff(2L)$,
\item[(iii)] $u = \ell_B u' r_B$ and $v = \ell_C v' r_C$ where
$|\widehat{u'}|, |\widehat{v'}| \ge \ff(2L)$.
\end{mylist}
Moreover, for all words $\eta,\theta \in \Sigma^*$ with
$|\eta|,|\theta| \le L$, we defined variables
$B_{\eta,\theta}$ and $C_{\eta,\theta}$ in $\cS$ whose values in $\cS$ are
$\nf(\eta u' \theta^{-1})$ and $\nf(\eta v' \theta^{-1})$, respectively.

We now define $\ell_A := \ell_B$, $r_A := r_C$, and $w' := u' r_B \ell_C v'$;
so that $w = \ell_A w' r_A$.
Note that, since $\tD(A) = \tD(B) = \tD(C) = \td$ and
$\Ht(B),\Ht(C) \le \Ht(A)$, we have (from (i) and (ii)) the length constraints
$4\ff(2L) \tD(A) \le |\widehat{\ell_A}|$, $|\widehat{r_A}| \le
(4\,\tD(A) + \Ht(A))\ff(2L)$. And from the lower bounds on $|\widehat{u'}|,
|\widehat{v'}|$ in (iii) we can deduce the required bound
$|\widehat{w'}|  \geq \ff(2L)$.

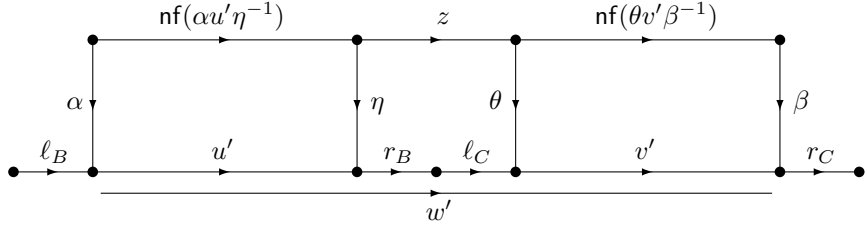
\begin{figure}
\begin{picture}(340,80)(0,-12)
%start with horizontal arrows and labels from bottom up
\put(43,-1){\vector(1,0){129}}
\put(170,-1){\line(1,0){127}}
\put(166,-11){$w'$}

\put(10,7){\circle*{4}}
\put(40,7){\circle*{4}}
\put(140,7){\circle*{4}}
\put(170,7){\circle*{4}}
\put(200,7){\circle*{4}}
\put(300,7){\circle*{4}}
\put(330,7){\circle*{4}}

\put(10,7){\vector(1,0){17}}
\put(25,7){\line(1,0){15}}
\put(40,7){\vector(1,0){52}}
\put(90,7){\line(1,0){50}}
\put(140,7){\vector(1,0){17}}
\put(155,7){\line(1,0){15}}
\put(170,7){\vector(1,0){17}}
\put(185,7){\line(1,0){15}}
\put(200,7){\vector(1,0){52}}
\put(250,7){\line(1,0){50}}
\put(300,7){\vector(1,0){17}}
\put(315,7){\line(1,0){15}}
\put(20,12){$\ell_B$}
\put(85,12){$u'$}
\put(150,12){$r_B$}
\put(180,12){$\ell_C$}
\put(245,12){$v'$}
\put(310,12){$r_C$}

\put(40,57){\circle*{4}}
\put(140,57){\circle*{4}}
\put(200,57){\circle*{4}}
\put(300,57){\circle*{4}}

\put(40,57){\vector(1,0){52}}
\put(90,57){\line(1,0){50}}
\put(140,57){\vector(1,0){32}}
\put(170,57){\line(1,0){30}}
\put(200,57){\vector(1,0){52}}
\put(250,57){\line(1,0){50}}
\put(65,63){$\nf(\alpha u' \eta^{-1})$}
%%%%%%
\put(170,63){$z$}
\put(230,63){$\nf(\theta v' \beta^{-1})$}

%now the vertical arrows and labels
\put(40,57){\vector(0,-1){27}}
\put(40,32){\line(0,-1){25}}
\put(140,57){\vector(0,-1){27}}
\put(140,32){\line(0,-1){25}}
\put(200,57){\vector(0,-1){27}}
\put(200,32){\line(0,-1){25}}
\put(300,57){\vector(0,-1){27}}
\put(300,32){\line(0,-1){25}}
\put(30,31){$\alpha$}
\put(145,31){$\eta$}
\put(190,31){$\theta$}
\put(305,31){$\beta$}
\end{picture}
\caption{Case 2.1}
\label{fig:case2.1}
\end{figure}

We already have \slps with values $\ell_A$ and $r_A$.
It remains to define the right-hand sides for the variables $A_{\alpha,\beta}$
for all $\alpha,\beta \in \Sigma^*$ with $|\alpha|,|\beta| \le L$.  

For each such $\alpha,\beta$, and
for all $\eta,\theta \in \Sigma^*$ with $|\eta|, |\theta| \le L$, we compute
(using Proposition \ref{prop:slpshort}) an \slp $\cZ$ with value
$z := \nf( \eta r_B \ell_C \theta^{-1} )$;
since $\hat{z}$ has length bounded by a constant multiple of $|V_\cU|$,
this computation takes time bounded by a polynomial in $|\cU|$,
and $|\cZ|$ is similarly bounded.
Then we check (in polynomial time)
using Proposition~\ref{prop:slp_results}(v) whether the word
\begin{eqnarray*}
\val_{\cS}(B_{\alpha,\eta}) \, z \, \val_{\cS}(C_{\theta,\beta}) &=&
   \nf(\alpha u' \eta^{-1}) \, \nf(\eta r_B \ell_C \theta^{-1}) \, 
\nf(\theta v' \beta^{-1})
\end{eqnarray*}
is $\nf$-reduced, in which case it is the word
$\nf(\alpha u' r_B \ell_C v' \beta^{-1})= \nf(\alpha w'\beta^{-1})$;
see Figure \ref{fig:case2.1}.

We claim that there must be at least one pair $\eta,\theta$ for which this
holds. To see this, we apply Proposition \ref{prop:backup2} twice. First
we apply it to the quadrilateral with sides labelled $\alpha$, $w'$, $\beta$,
$\nf(\alpha w' \beta^{-1})$, using $|\widehat{u'}|,
|\widehat{v'}| \ge \ff(2L) \ge \ff(L)$, to define $\eta$, and then to the
quadrilateral with sides labelled $\eta$, $r_B\ell_Cv'$, $\beta$,
$\nf (\eta r_B\ell_Cv'\beta^{-1})$ using $|\widehat{r_B\ell_C}|,
|\widehat{v'}| \ge \ff(L)$ to define $\theta$. 
Proposition~\ref{prop:backup2} ensures that $|\theta|, |\eta|\leq L'$, and since
in Section~\ref{subsec:relhyp_param2} we chose $L\geq L'$, we certainly
have $|\theta|,|\eta| \leq L$.
We then include the variables of the \slp $\cZ$ within $\cS$ and 
define $\rho_\cS(A_{\alpha,\beta}) := B_{\alpha,\eta} \, S_\cZ \,
C_{\theta,\beta}$.

Suppose next (Case 2.2) that
$|\hat{u}| > (8\td + 1)\ff(2L)$ and $|\hat{v}| \le (8\td + 1)\ff(2L)$
(so again we do not have $uv \in \Sigma_i^*$).
(Case 2.3 where
$|\hat{u}| \le (8\td + 1)\ff(2L)$ and $|\hat{v}| > (8\td + 1)\ff(2L)$
is similar, and we shall omit the details.)
Then we have already computed an \slp with value $v$, and
\slps for words $\ell_B$ and $r_B$ where:
\begin{mylist}
\item[(i)]  $4\td \ff(2L) \le |\widehat{\ell_B}|, |\widehat{r_B}| \le
  (4\td + \Ht(B))\ff(2L)$,
\item[(ii)] $u = \ell_B u' r_B$ for a word $u'$ with
$|\widehat{u'}| \ge \ff(2L)$.
\end{mylist}
Moreover, for all words $\eta,\theta \in \Sigma^*$ with
$|\eta|,|\theta| \le L$, we have defined variables
$B_{\eta,\theta}$ with value $\nf(\eta u' \theta^{-1})$.  

%%%%%%%
If $|\hat{v}| \leq \ff(2L)$, then we set $\ell_A := \ell_B$, $w':=u'$, and
$r_A := r_B v$, 
so that $A_\ell=B_\ell$ and $A_r$ is the root of an \slp with value $r_Bv$,
and we define $\rho_\cS(A_{\alpha,\beta}) := B_{\alpha,\beta}$ for all
$\alpha,\beta \in \Sigma^*$ with $|\alpha|,|\beta| \le L$; so
the step takes polynomial time and adds right-hand sides of bounded total
length.  In this case, we have
$4\td \ff(2L) \le |\widehat{\ell_A}| \le (4\td + \Ht(B))\ff(2L) \le
(4\td + \Ht(A))\ff(2L)$ and
$4\td \ff(2L) \le |\widehat{r_A}| \le (4\td + \Ht(B) + 1)\ff(2L) \le
(4\td + \Ht(A))\ff(2L)$, as required.

So now assume that $|\hat{v}| > \ff(2L)$. Again, we set $\ell_A := \ell_B$.
Since we are assuming that $\cU$ is non-splitting and $\rho(A) = BC$,
an occurrence of $\val_\cU(B)$ as a prefix of an occurrence of $\val_\cU(A)$
cannot split a component, and so we have
$\widehat{r_Bv}= \widehat{r_B}\hat{v}$. 
%%%%%
Since $|\widehat{r_B v}| \ge (4\td +1)\ff(2L)$ we can define $r_A$ as the
suffix of $r_B v$ with $|\widehat{r_A}| = 4\td \ff(2L)$;
that is, $r_B v = y r_A$ for some word $y$.

Since $\widehat{r_Bv}=\widehat{r_B}\hat{v}$ and $|\widehat{r_B}| \ge
4\td \ff(2L) = |\widehat{r_A}|$, we have
$|\hat{y}| = |\widehat{r_B}|+|\hat{v}|-|\widehat{r_A}| \ge
|\hat{v}| > \ff(2L)$. 
(Recall that $r_B$ and $v$ are in normal form, and hence $\widehat{r_B}$ and $\hat{v}$ are geodesic.)
Then $w = \ell_Aw'r_A$ with $w' = u'y$.
This satisfies the required bounds
on the lengths of $\widehat{\ell_A}$, $\widehat{r_A}$ and $\widehat{w'}$.
We can use Proposition \ref{prop:comproot} and Corollary \ref{cor:gamlen}
to define \slps with values $y$ and $r_A$,
in time polynomial in $|\cU|$; we set $A_r$ to be the root of the second of these. 

\begin{figure}
\begin{picture}(240,80)(-10,-12)
%first horizontal arrows and labels
\put(173,-1){\vector(1,0){29}}
\put(200,-1){\line(1,0){27}}
\put(196,-11){$v$}

\put(10,7){\circle*{4}}
\put(40,7){\circle*{4}}
\put(140,7){\circle*{4}}
\put(170,7){\circle*{4}}
\put(200,7){\circle*{4}}
\put(230,7){\circle*{4}}

\put(10,7){\vector(1,0){17}}
\put(25,7){\line(1,0){15}}
\put(40,7){\vector(1,0){52}}
\put(90,7){\line(1,0){50}}
\put(140,7){\vector(1,0){17}}
\put(155,7){\line(1,0){15}}
\put(170,7){\vector(1,0){17}}
\put(185,7){\line(1,0){15}}
\put(200,7){\vector(1,0){17}}
\put(215,7){\line(1,0){15}}
\put(20,-3){$\ell_B$}
\put(85,-3){$u'$}
\put(150,-3){$r_B$}

\put(143,15){\vector(1,0){29}}
\put(170,15){\line(1,0){27}}
\put(166,20){$y$}

\put(40,57){\circle*{4}}
\put(140,57){\circle*{4}}
\put(200,57){\circle*{4}}

\put(40,57){\vector(1,0){52}}
\put(90,57){\line(1,0){50}}
\put(140,57){\vector(1,0){32}}
\put(170,57){\line(1,0){30}}
\put(65,63){$\nf(\alpha u' \eta^{-1})$}
\put(165,63){$z$}

%now the vertical arrows and labels
\put(40,57){\vector(0,-1){27}}
\put(40,32){\line(0,-1){25}}
\put(140,57){\vector(0,-1){27}}
\put(140,32){\line(0,-1){25}}
\put(200,57){\vector(0,-1){27}}
\put(200,32){\line(0,-1){25}}
\put(30,31){$\alpha$}
\put(145,31){$\eta$}
\put(205,31){$\beta$}
\end{picture}
\caption{Case 2.2}
\label{fig:case2.2}
\end{figure}
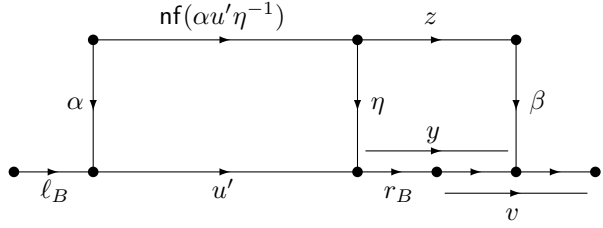

It remains to define the right-hand sides for the variables $A_{\alpha,\beta}$
for all words $\alpha,\beta \in \Sigma^*$ with $|\alpha|, |\beta| \le L$.
Now, for all $\eta \in \Sigma^*$ with $|\eta| \le L$, we compute
(using Proposition \ref{prop:slpshort}) an \slp $\cZ$ with value
$z := \nf( \eta y \beta^{-1})$; %since $z$ is short
this is done in polynomial time and $\cZ$ has polynomially bounded size.
Then we check whether the word
\[ \val_{\cS}(B_{\alpha,\eta}) \, z =
\nf(\alpha u' \eta^{-1}) \, \nf(\eta y \beta^{-1})
\]
is $\nf$-reduced, in which case it is
$\nf(\alpha u' y \beta^{-1}) = \nf(\alpha w'\beta^{-1})$;
see Figure \ref{fig:case2.2}.
Since $|\hat{u}|>(8\td+1)\ff(2L)>\ff(L)$ and 
$|\hat{y}| > \ff(2L) \ge \ff(L)$, Proposition~\ref{prop:backup2} implies
that there must be at least one such $\eta$ 
with $|\eta| \leq L'\leq L$. 
We include the variables of the
\slp $\cZ$ within $\cS$ and define
$\rho_\cS(A_{\alpha,\beta}) := B_{\alpha,\eta} \, S_\cZ$.

Finally (Case 2.4), suppose that
$|\hat{u}|, |\hat{v}| \le (8\td + 1)\ff(2L)$
(and hence $|\hat{w}| \le 2(8\td+1)\ff(2L)$).
Note that we could now have $w \in \Sigma_i^*$ for some $i$.
In this case, we have already computed \slps with values $u$ and $v$
when we processed $B,C$.
If $|\hat{w}| \le (8\td+ 1)\ff(2L)$ (which is certainly true when
$w \in \Sigma_i^*$) then we can define an \slp
for $w$ as the concatenation of those for $u$ and $v$
and attach it to $\cS$ with $A$ as its root.
We can then use Proposition \ref{prop:slpshort} to ensure that this
\slp satisfies the required bound on its size.

%%%%%%%%%%%%%%%%%%%%%%%%%%%%%%%%%%%%
Otherwise we factorise $w$ as $w = \ell_A w' r_A$ with
$|\widehat{\ell_A}| = |\widehat{r_A}| = 4\td \ff(2L) $, and thus
$(8\td+ 2)\ff(2L) \ge |\widehat{w'}| \ge \ff(2L)$.
We use Proposition \ref{prop:comproot} and Corollary \ref{cor:gamlen}
to define \slps with values $\ell_A$ and $r_A$, and attach those to $\cS$. 
For $\alpha,\beta \in \Sigma^*$ with $|\alpha|,|\beta| \le L$ we compute
(using Proposition \ref{prop:slpshort}), an \slp $\cZ$ with value
$\nf(\alpha w' \beta^{-1})$;
Then we attach $\cZ$ within $\cS$ with $A_{\alpha,\beta}$ as its root,
so that $\val_\cS(A_{\alpha,\beta}) = \nf(\alpha w' \beta^{-1})$.
This computation takes time bounded by a polynomial function of its
input size which, as a result of our size restriction on the \slps
$\cS_B$ and $\cS_C$ is in turn bounded by a polynomial function of $|\cU|$;
so the same applies to the size of $\cZ$.

{\bf Case 3.} $\rho_\cU(A) = B\langle \sigma,\tau \rangle$ for a variable
$B$ of $\cU$ and words $\sigma,\tau \in \Sigma^*$ with $|\sigma|,|\tau| \le L$.
Let $u := \val_{\cU}(B)$ and $w := \val_{\cU}(A) = \nf(\sigma u \tau^{-1})$.
Let $\td := \tD(B)$. Then $\tD(A) = \td-1 \ge 1$, and so $\td \geq 2$.

If $|\hat{u}| \le (8\td +1) \ff(2L)$ (Case 3.1) ,
then we have already computed an \slp with value $u$. 
In that case, we proceed much as we did in Case~2.4, as follows.
Using  Proposition \ref{prop:slpshort},
we construct an \slp for $\sigma u \tau^{-1}$ as the concatenation of
three \slps (the word $\sigma$, our \slp for $u$, and the word $\tau$),
and then construct from this an \slp $\cZ$ for $w$.
This computation takes time bounded by $\jj((8\td+1)\ff(2L)+2L)$ and $\cZ$
has size bounded by
$\CC|\hat{w}|\log(|w|)$.
If $|\hat{w}| \le (8 \td + 1)\ff(2L)$,
then we just attach $\cZ$ to $\cS$ with $A$ as its root.

Otherwise, we factorise $w$ as $w=\ell_Aw'r_A$ and, just as in the second part
of Case~2.4, we define \slps with values $\ell_A$ and $r_A$, and attach those
to $\cS$ and, for each $\alpha,\beta \in \Sigma^*$ with
$|\alpha|,|\beta| \leq L$, we compute an \slp with value
$\nf(\alpha w'\beta^{-1})$, and attach it to $\cS$ with the new variable
$A_{\alpha,\beta}$ as its root, so that
$\val_{\cS}(A_{\alpha,\beta}) = \nf(\alpha w' \beta^{-1})$.

Otherwise (Case 3.2), we have  $|\hat{u}| > (8\td + 1) \ff(2L)$.
We have computed \slps for words $\ell_B, r_B$ such that
$4\td \ff(2L) \le |\widehat{\ell_B}|, |\widehat{r_B}| \le
(4\td + \Ht(B))\ff(2L)$
and $u = \ell_B u' r_B$ for a word $u'$ with $|\widehat{u'}| \ge \ff(2L)$.
Moreover, for all words $\eta,\theta \in \Sigma^*$ with
$|\eta|,|\theta| \le L$, we have defined variables $B_{\eta,\theta}$
with value $\nf(\eta u' \theta^{-1})$.

We check for all $\eta,\theta \in \Sigma^*$ with $|\eta|,|\theta| \le L$
whether
\[
\nf(\sigma \ell_B \eta^{-1}) \val_{\cS}(B_{\eta,\theta})
\nf(\theta r_B \tau^{-1}) = \nf(\sigma \ell_B \eta^{-1})
  \nf(\eta u' \theta^{-1}) \nf(\theta r_B \tau^{-1})
\]
is $\nf$-reduced, in which case it is
$\nf(\sigma \ell_B u' r_B \tau^{-1})  = \nf(\sigma u \tau^{-1}) =w$.
Since $|\widehat{\ell_B}|,|\widehat{r_B}|,|\widehat{u'}| \ge \ff(L)$,
we can show that such words $\eta,\theta$ exist by two applications of
Proposition~\ref{prop:backup2} as we did in Case 2.1.

Let $s := \nf(\sigma \ell_B \eta^{-1})$ and
$t := \nf(\theta r_B \tau^{-1})$,
so that $w=s\,\nf(\eta u'\theta^{-1}) t$.  
We deduce from the triangle inequality
that $|\widehat{\ell_B}| \leq |\widehat{\sigma}|+|\hat{s}|+|\hat{\eta}| \leq |\hat{s}|+2L$
and similarly
that $|\widehat{r_B}| \leq |\hat{t}|+2L$,
and hence
we have $|\hat{s}|,\, |\hat{t}| \ge 4\td \ff(2L) -2L \ge
(4 \td-1) \ff(2L)\geq 7\ff(2L)$.  (Recall that $\td\geq 2$.)
Hence we can factorise these words as $s = vx$ and $t = yz$ with
$|\hat{v}| = |\hat{z}| = 4(\td-1) \ff(2L) =
4\tD(A) \ff(2L) \ge 4\ff(2L)$,
and $|\hat{x}|, \,|\hat{y}| \ge 3\ff(2L)$.

We define $\ell_A := v$ and $r_A := z$;
these words satisfy the required bounds on their lengths.  Note that
$\val_{\cU}(A) = \nf(\sigma u \tau^{-1}) = \ell_A w' r_A$ with
$w' := x \, \nf(\eta u' \theta^{-1}) y$,
and $|\widehat{w'}| \ge 6\ff(2L) \ge \ff(2L)$.
As in  earlier cases we can apply Proposition \ref{prop:comproot} and
Corollary \ref{cor:gamlen} to compute \slps with values $v=\ell_A$,
$x$, $y$, $z=r_A$, and $w'$.

\begin{figure}
\begin{picture}(320,120)(-20,0)
%start with horizontal arrows and labels from bottom up
\put(10,7){\vector(1,0){42}}
\put(50,7){\line(1,0){40}}
\put(230,7){\vector(1,0){42}}
\put(270,7){\line(1,0){40}}
\put(45,-3){$\ell_B$}
\put(265,-3){$r_B$}

\put(10,15){\circle*{4}}
\put(50,15){\circle*{4}}
\put(90,15){\circle*{4}}
\put(230,15){\circle*{4}}
\put(270,15){\circle*{4}}
\put(310,15){\circle*{4}}

\put(10,15){\vector(1,0){22}}
\put(30,15){\line(1,0){20}}
\put(50,15){\vector(1,0){22}}
\put(70,15){\line(1,0){20}}
\put(230,15){\vector(1,0){22}}
\put(250,15){\line(1,0){20}}
\put(270,15){\vector(1,0){22}}
\put(290,15){\line(1,0){20}}
\put(90,15){\vector(1,0){72}}
\put(160,15){\line(1,0){70}}
\put(23,20){$v'$}
\put(65,20){$x'$}
\put(243,20){$y'$}
\put(285,20){$z'$}
\put(155,20){$u'$}

\put(10,65){\circle*{4}}
\put(50,65){\circle*{4}}
\put(90,65){\circle*{4}}
\put(230,65){\circle*{4}}
\put(270,65){\circle*{4}}
\put(310,65){\circle*{4}}

\put(10,65){\vector(1,0){22}}
\put(30,65){\line(1,0){20}}
\put(50,65){\vector(1,0){22}}
\put(70,65){\line(1,0){20}}
\put(230,65){\vector(1,0){22}}
\put(250,65){\line(1,0){20}}
\put(270,65){\vector(1,0){22}}
\put(290,65){\line(1,0){20}}
\put(90,65){\vector(1,0){72}}
\put(160,65){\line(1,0){70}}
\put(23,55){$v$}
\put(65,55){$x$}
\put(243,55){$y$}
\put(285,55){$z$}

\put(53,73){\vector(1,0){109}}
\put(160,73){\line(1,0){107}}
\put(155,78){$w'$}

\put(50,115){\circle*{4}}
\put(90,115){\circle*{4}}
\put(230,115){\circle*{4}}
\put(270,115){\circle*{4}}

\put(50,115){\vector(1,0){22}}
\put(70,115){\line(1,0){20}}
\put(230,115){\vector(1,0){22}}
\put(250,115){\line(1,0){20}}
\put(90,115){\vector(1,0){72}}
\put(160,115){\line(1,0){70}}

%now vertical arrows and labels from bottom up
\put(10,65){\vector(0,-1){27}}
\put(10,40){\line(0,-1){25}}
\put(0,37){$\sigma$}
\put(50,65){\vector(0,-1){27}}
\put(50,40){\line(0,-1){25}}
\put(40,37){$\mu$}
\put(90,65){\vector(0,-1){27}}
\put(90,40){\line(0,-1){25}}
\put(80,37){$\eta$}

\put(230,65){\vector(0,-1){27}}
\put(230,40){\line(0,-1){25}}
\put(235,37){$\theta$}
\put(270,65){\vector(0,-1){27}}
\put(270,40){\line(0,-1){25}}
\put(275,37){$\nu$}
\put(310,65){\vector(0,-1){27}}
\put(310,40){\line(0,-1){25}}
\put(315,37){$\tau$}

\put(50,115){\vector(0,-1){27}}
\put(50,90){\line(0,-1){25}}
\put(40,87){$\alpha$}
\put(270,115){\vector(0,-1){27}}
\put(270,90){\line(0,-1){25}}
\put(275,87){$\beta$}

%finally those curly lines
\qbezier(90,115)(120,65)(90,15)
\put(104.5,55){\vector(0,-1){0}}
\put(108,53){$\chi$}
\qbezier(230,115)(200,65)(230,15)
\put(215.3,55){\vector(0,-1){0}}
\put(204,53){$\psi$}

\end{picture}
\caption{Case 3}
\label{fig:case3}
\end{figure}

It remains to define the right-hand sides of the variables $A_{\alpha,\beta}$
with values $\nf(\alpha w' \beta^{-1})$
for all words $\alpha,\beta \in  \Sigma^*$ with $|\alpha|,|\beta| \le L$.
The lower bounds on the lengths of
$\hat{v},\hat{x},\hat{y},\hat{z}$ allow us to apply
Proposition \ref{prop:backup2} to the quadrilaterals with sides
labelled $\sigma, \ell_B, \eta, vx$ and $\theta, r_B, \tau, yz$, respectively.
We can compute in polynomial time  words $\mu,\nu \in  \Sigma^*$ with
$|\mu|,|\nu| \le L$ and factorisations $\ell_B = v'x'$, $r_B = y'z'$ such
that $\sigma v' =_G v \mu$, $x \eta =_G \mu x'$, $\theta y' =_G y\nu$, and
$\nu z' =_G z \tau$. 
By the triangle inequality, the words $x'$ and $y'$
satisfy  $|\widehat{x'}|,\,|\widehat{y'}| \ge \ff(2L)$.

Now consider the quadrilateral with sides labelled
$x'u'y'$, $\nf(\alpha\mu)$, $\nf(\beta\nu)$, and
$\nf(\alpha\mu x'u'y'\nu^{-1}\beta^{-1})$.
Since  $|\widehat{x'}|,\,|\widehat{y'}|, |\widehat{u'}| \ge \ff(2L)$  and
$|\nf(\alpha\mu)|, |\nf(\nu\beta)| \leq 2L$, we can again make two
applications of Proposition \ref{prop:backup2} to show that
there exist words $\chi,\psi \in  \Sigma^*$ with $|\chi|,\,|\psi| \le L' \le L$
such that the word
\begin{eqnarray*}
\lefteqn{\nf(\alpha\mu x' \chi^{-1})\, \val_{\cS}(B_{\chi,\psi}) \,
\nf(\psi y' \nu^{-1}\beta^{-1}) =}\\
&& \nf(\alpha\mu x' \chi^{-1})\,
\nf(\chi u' \psi^{-1}) \,\nf(\psi y' \nu^{-1}\beta^{-1})
\end{eqnarray*}
is $\nf$-reduced, in which case the above word is
$\nf(\alpha\mu x'  u'  y' \nu^{-1}\beta^{-1}) = \nf(\alpha w' \beta^{-1})$;
see Figure \ref{fig:case3}.
As before, we can find these words $\chi,\psi$ in polynomial time.
Finally (using Proposition \ref{prop:slpshort}), we define
\slps $\cY$ and $\cZ$ with values $\nf(\alpha\mu x' \chi^{-1})$ and
$\nf(\psi y' \nu^{-1}\beta^{-1})$, include their variables within $\cU$, and then
define $\rho_\cS(A_{\alpha,\beta}) := S_\cY\, B_{\chi,\psi} \, S_\cZ.$
This concludes the definition of the right-hand sides for the variables
$A_{\alpha,\beta}$.

We complete the definition of the \slp $\cS$ by adding a start variable
$S_\cS$ to $\cS$ and setting
$\rho_\cS(S_\cS) := S_\ell S_{\emptyword,\emptyword} S_r$,
where $S := S_\cU$ is the start variable of $\cU$ and $S_\ell$ and $S_r$
are the variables with values $\ell_S$ and $r_S$.
This ensures $\val(\cS) = \ell_S \, \nf(s')\, r_S$, where $s'$ is such that
$\ell_S s' r_S = \val(\cU)$.
But we are assuming that $\cU$ is $\nf$-reduced, so
$s'$ is also $\nf$-reduced and we get
$\val(\cS) = \ell_S \, \nf(s') \, r_S = \ell_S s' r_S = \val(\cU)$.
\end{proof}

\begin{proposition}\label{prop:tcslp-tslp}
Let $G,\Sigma$ and $L$ be as in Theorem~\ref{thm:convert}. Then,
given an $\nf$-reduced non-splitting \tcslp $\cT$ for $G$ 
over $\Sigma$ with $\JJ_\cT \le L$,
such that each of its cut operators is non-splitting and
specified relative to compression,
we can compute in polynomial time an $\nf$-reduced non-splitting \tslp $\cU$
with $\JJ_{\cU} \le L$ and $\val(\cU) = \val(\cT)$,
whose size is bounded by a polynomial function of $|\cT|$.
\end{proposition}
\begin{proof}
We follow the proof of \cite[Lemma 6.5]{HLS}.
The idea of the proof is taken from \cite{Hag00}, where it is shown that a
\cslp can be transformed in polynomial time into an \slp with the same value.
Let $\cT = (V_\cT,S_\cT,\rho_\cT)$ be the input TCSLP.
As discussed in Section~\ref{subsec:extend_slps},
we can assume that all right-hand sides from $(V_\cT \cup \Sigma)^*$ are
of the form $a \in \Sigma$ or $BC$ with $B,C \in V_\cT$.

Consider a variable $A$ such that $\rho_\cT(A) = B[[:i))$; the case that
$\rho_\cT(A) = B[[i:))$ can be dealt with analogously.
By considering the variables in order of increasing height,
we can assume that no cut operator occurs in the right-hand side of any
variable $A'$ with $\Ht(A') < \Ht(A)$.
Using Proposition \ref{prop:tslp-slp} we can compute
an \slp with value $\val_{\cT}(B)$ and then use
Corollary~\ref{cor:gamlen}
to compute $n_B:= |\widehat{\val_{\cT}(B)}|$ in polynomial time.

Now we show how to eliminate the cut operator in $\rho_\cT(A)$. This involves
adding at most $\Ht(\cT)$ new variables to the TCSLP. Moreover
the height of the \tcslp after the cut elimination will still be bounded by
$\Ht(\cT)$. Hence, the final \tslp has at most $\Ht(\cT) \cdot |V|$
variables, and its size is polynomially bounded.

The idea of the cut elimination is to push the cut operator towards
variables of lesser height.
For this, we need to consider the various possibilities
for the right-hand side of $B$.

{\bf Case 1.} $\rho_\cT(B) = a \in \Sigma$. If $i = 1$ we define
$\rho_\cU(A) := a$, and if $i = 0$ we define $\rho_\cU(A) := \emptyword$.

{\bf Case 2.} $\rho_\cT(B) = CD$ with $C,D \in V$.
Since we are assuming that the cut operator in $\rho_\cT(A) = B[[i:))$
is non-splitting, $\val_\cU(B)$ cannot consist of a single 
component, and so $\val_\cU(C)$ and $\val_\cU(D)$ must consist of complete
components of $\val_\cU(B)$ together with subwords in $(\Sigma \setminus \cH)^*$.
Define $n_C := |\widehat{\val_\cT(C)}|$.
If $i \leq n_C$ then we define $\rho_\cU(A) := C[[:i))$.
If $i > n_C$ then we define $\rho_\cU(A) := CX$ for a new variable $X$  and set
$\rho_\cU(X) := D[[:i-n_C))$.
We then continue with the elimination of the remaining cut operator in $C[[:i))$
or in $D[[:i-n_C))$.

{\bf Case 3.} $\rho_\cT(B) = C\langle \alpha,\beta \rangle$ with $C \in V$ and
$\alpha,\beta \in \Sigma^*$ with $|\alpha|,|\beta| \le \JJ_{\cT} \le L$.
Let $u := \val_{\cT}(C)$, $v := \val_{\cT}(B)=_G \alpha u\beta^{-1}$, and
$v_1=\val_{\cT}(A)$. So $|\widehat{v_1}| = i$
and, where we write $v=v_1v_2$, our condition
on the non-splitting of cut operators ensures that not only $v_1$ but also
$v_2$ must consist of complete components of $v$ together with subwords in
$(\Sigma \setminus \cH)^*$; so $|\widehat{v_2}|=n_B-i$.
Thus, we have $\val_{\cT}(A) = v_1$ and $v = \nf(\alpha u \beta^{-1})$.
By Proposition~\ref{prop:tslp-slp} we can assume that we have computed \slps
with values  $u$ and $v$ and we may assume by Proposition \ref{prop:comproot}
that all components of $u$ and $v$ have roots in these \slps.

Suppose first that $|\widehat{v_1}|,|\widehat{v_2}| \ge \ff(L)$.
(This corresponds to Case 3.3 of the proof of \cite[Lemma 6.5]{HLS}.)
Then by Proposition \ref{prop:backup2} there exists a factorisation
$u = u_1 u_2$ and $\eta \in \Sigma^*$ with $|\eta| \le L$ such that
$v_1 =_G \alpha u_1 \eta^{-1}$ and $v_2 =_G \eta u_2 \beta^{-1}$.
We note that $|\hat{\alpha}|$, $|\hat{\beta}|$, $|\hat{\eta}|\le L$,
and applying the triangle inequality in a quadrilateral
within $\widehat{\Gamma}$ with geodesic sides labelled by
$\hat{\alpha}$, $\widehat{v_1}$, $\hat{\eta}$, $\widehat{u_1}$,
we see that $i-2L \leq |\widehat{u_1}| \leq i + 2L$,
and so we can find such a factorisation of $u$ in polynomial time
by computing \slps for the words $u[[:j))$ and $u[[j:))$
for every integer $j$ with $|i-j|\leq 2L$.  Then we apply
Proposition~\ref{prop:tslp-slp} and compute \slps for the words
$w_1 := \nf(\alpha u[[:j)) \eta^{-1})$ and
$w_2 := \nf(\eta u[[j:)) \beta^{-1})$
for every $\eta \in \Sigma^*$ with $|\eta| \le L$.
Finally we can, by using Proposition \ref{prop:slp_results}\,(vi),
check whether $v_1 = w_1$ and $v_2 = w_2$
(we are guaranteed to find at least one such $j$ and $\eta$),
and we define $\rho_\cU(A) := X \langle \alpha,\eta\rangle$
for a new variable $X$ and set $\rho_\cU(X) := C[[:j))$.
We then continue with the elimination of the cut operator in $C[[:j))$.

The other two cases are a little more complicated than in \cite{HLS}. Suppose
first that  $|\widehat{v_1}| < \ff(L)$. Then by Proposition \ref{prop:slpshort},
in polynomial time we can compute an \slp $\cT_A$ with
$\val(\cT_A) = \val(A) = v_1$. We include the variables and productions of
$\cT_A$ as part of the \tslp $\cU$ that we are constructing, and
define $\rho_\cU(A) := S_{\cT_A}$.

Otherwise we have $|\widehat{v_1}| \ge \ff(L)$ and $|\widehat{v_2}| < \ff(L)$.
Let $k:= |\hat{v}| - \ff(L)$. Then, as above, we can
use Proposition \ref{prop:tcslp-tslp} together with
Proposition \ref{prop:backup2} to find $l$ and a word $\eta \in \Sigma^*$ with
$|\eta| \le L$ such that $v[[:k)) =_G \alpha u[[:l)) \eta^{-1}$.
Now we introduce
a new variable $X$ with $\rho_\cU(X) = C[[:l))$.
But in addition, using  Proposition \ref{prop:slpshort}
we compute, in polynomial time,
an \slp $\cT_A$ with $\val(\cT_A) = v[[k:i))$, include
the variables and productions of $\cT_A$ in $\cU$, and
and define $\rho_\cU(A) := X \langle \alpha,\eta \rangle S_{\cT_A}$.
Then $\val(A) = v[[:k))v[[k:i)) = v[[:i)) = v_1$ as required.
As before, we continue with the elimination of the cut operators below 
$C[[:l))$.

Since each elimination of a cut operator in $\rho(A)$ can lead to further
such eliminations in the variables below $A$, the total number of such
eliminations in the processing of $\cT$ is bounded by $\Ht(\cT)^2$.
\end{proof}

With the completion of the proofs of Propositions~\ref{prop:tslp-slp} and 
~\ref{prop:tcslp-tslp} we have now completed the proof of 
Theorem~\ref{thm:convert}.
The following corollary of that theorem will be used in the final section.
\begin{corollary}\label{cor:addtail}
Let $\cG$ be an \slp over $\Sigma$ for which $w := \val(\cG)$ is $\nf$-reduced,
and let $v \in \Sigma^*$ have length at most $L$.  Then we can, in polynomial
time, compute an \slp $\cS$ over $\Sigma$ with $\val(\cS) = \nf(wv)$.
\end{corollary}
\begin{proof} From $\cG$, we can immediately define a \tslp $\cT$ 
with $\val(\cT) = \nf(wv)$ and $\JJ_\cT = L$ by adjoining a new start
variable $S_\cT$ together with the production
$\rho(S_\cT) = S_\cG \langle \emptyword,v^{-1} \rangle$. We
can then construct $\cS$ with $\val(\cS) = \nf(wv)$ in polynomial time,
by Theorem~\ref{thm:convert}.
\end{proof}

%%%%%%%%%%%%%%%%%%%%%%%%%%%%
\section{The final step.}
\label{sec:slextcslp}
%%%%%%%%%

Since a word represents the identity element if and only if its $\nf$-reduction
is the empty word, the main theorem, Theorem~\ref{thm:main},
follows immediately from the combination of the following
Theorem~\ref{thm:slextcslp} with Theorem~\ref{thm:convert}.
Hence the proof of Theorem~\ref{thm:slextcslp} is our final step.

\begin{theorem}\label{thm:slextcslp}
Let $G$ be a group hyperbolic relative to a collection
of free abelian subgroups, and suppose that a generating set $\Sigma$ for $G$,
and integer $L$ are selected as in
Sections ~\ref{subsec:relhyp_param1}, \ref{subsec:relhyp_param2}.
Let $\cG$ be an \slp for $G$ over $\Sigma$.
Then we can construct, in polynomial time, a non-splitting
$\nf$-reduced \tcslp $\cT$ with $\val(\cT) = \nf(\val(\cG))$ and
$\JJ_\cT \le L$, where each cut operator of $\cT$ is non-splitting and
specified relative to compression.
\end{theorem}

The following lemma, which is a special case of Theorem ~\ref{thm:slextcslp},
will be used within its proof, and applied to sub-\slps of the \slp $\cG$ within
the statement of the theorem.  Since the proof of the lemma involves similar but
simpler arguments to that of the theorem, it is convenient to defer its proof.  
\begin{lemma}\label{lem:geo2slextcslp}
Let $G$, $\Sigma$ and $\cG$ be as in Theorem ~\ref{thm:slextcslp},
and assume also that $\widehat{\val(\cG)}$ is a geodesic word.
Then in polynomial time we can construct a non-splitting
$\nf$-reduced \tcslp $\cT$ with $\val(\cT) = \nf(\val(\cT))$  and
$\JJ_\cT \le L$, where each cut operator of $\cT$ is non-splitting and
specified relative to compression.
\end{lemma}

\begin{proofof}{Theorem~\ref{thm:slextcslp}}
We know from Proposition \ref{prop:genset}
that $G$ is asynchronously automatic over $\Sigma$ so, by
Proposition~\ref{prop:slp_results}\,(v), we can test in polynomial time
whether the words
defined by \slps over $\Sigma$ are $\nf$-reduced.
We may assume by Proposition~\ref{prop:slp_results} that the given
\slp $\cG= (V_\cG,S_\cG,\rho_\cG)$ is trimmed and in Chomsky normal form,
and by Proposition \ref{prop:comproot} that all components of
$\val(\cG)$ have roots.

The proof follows the same strategy as that of \cite[Theorem 6.7]{HLS},
but the presence of the parabolic subgroups gives rise to complications.
Our aim is to construct a \tcslp $\cT=(V_\cT,S_\cT,\rho_\cT)$ with value
$\nf(\val(\cG))$ that satisfies $\JJ_\cT \le L$,
where $L$ is the constant defined in Section \ref{subsec:relhyp_param2}.

We consider the variables of $\cG$ in order of increasing height;
$V_\cT$ will contain a copy $A$ of each variable $A$ of $\cG$,
together with some auxiliary variables.
We build $\cT$ piece by piece, starting with $V_\cT$ empty.
At each stage, as we consider the variable $A$,
we add  a copy of $A$ and possibly some other new variables to $V_\cT$,
and define the image of $\rho_\cT$ on each of those so as to make
$\val_{\cT}(A) = \nf(\val_\cG(A))$.
Our construction of \tcslps rather than \cslps will ensure a polynomial bound
on the number of new variables we add at each stage and on the size of $\cT$.
That the condition on cut-operators holds will be clear from the construction.

If the variable $A$ under consideration has height one,
then $\rho_\cG(A)=a$, for some $a \in \Sigma$;
in that case, we simply define $\rho_\cT(A)= \nf(a) = a$.
So from now on we suppose that $\Ht(A)>1$, in which case
$\rho_\cG(A)=BC$, for variables $B$, $C$ of height less than $\Ht(A)$.
Since we  have already processed the variables $B$ and $C$,
we know that $\cT$ already contains sub-\tcslps $\cT_B$
and $\cT_C$, with start variables $B_\cT$ and $C_\cT$, and with
$\val(\cT_B) = v_1 := \nf(\val_\cG(B))$,
$\val(\cT_C) = v_2 := \nf(\val_\cG(C))$,
and $\JJ_{\cT_B},\JJ_{\cT_C} \le L$.

By Theorem~\ref{thm:convert}, we can construct, in polynomial
time, \slps $\cS_B$ and $\cS_C$, with the same values as $\cT_B$ and $\cT_C$.
As was the case for $\cG$,
we can ensure that $\cS_B$ and $\cS_C$ are in Chomsky normal form,
and contain roots for all components of $v_1$ and $v_2$, respectively.

We consider a triangle in the Cayley graph $\Gamma = \Gamma(G,\Sigma)$
whose sides are paths labelled by $v_1$, $v_2$ and the $\nf$-reduced
representative $v_3$ of the element $v_1v_2$. Our aim is to
construct $\cT_A$ as a \tcslp with value $v_3$.
Let $a,b$ and $c$ be the vertices of this triangle
with sides from $a$ to $b$, $b$ to $c$ and
$a$ to $c$ labelled by $v_1,v_2,v_3$. So the sides labelled
$\widehat{v_1},\widehat{v_2},\widehat{v_3}$ in the corresponding triangle
in $\widehat{\Gamma}$ are geodesic.

So, since $\widehat{\Gamma}$ is a $\delta$-hyperbolic space,
there are meeting vertices $d_1,d_2,d_3$, with $d_i$ on
$\gamma_{\widehat{v_i}}$ ($i=1,2,3$) and
$d_{\widehat{\Gamma}}(d_i,d_j) \leq \delta$ for $i \neq j$;
see Figure~\ref{fig:hyptri}.
Now let $K := K_1(\delta)$ as defined in Proposition~\ref{prop:backup}, 
and recall that the constant $L$ defined in
Section \ref{subsec:relhyp_param2} satisfies $L \ge L_1(\delta)$.  Now we apply
Proposition \ref{prop:backup}
to the sections of $\gamma_{\widehat{v_1}}$ and
$\gamma_{\widehat{v_2}}$ that join $b$ to $d_1$ and $d_2$,
and so deduce that,
for any vertex $b_1$ of $\widehat{\Gamma}$ on
$\gamma_{\widehat{v_1}}$ between $d_1$ and $b$ 
and distance on $\gamma_{\widehat{v_1}}$ at least $K$ from $d_1$,
there exists a vertex $b_2$, on $\gamma_{\widehat{v_2}}$,
with $d_{\Gamma}(b_1,b_2) \leq L$.
We shall call vertices $b_1,b_2$ of $\widehat{\Gamma}$ with
$d_{\Gamma}(b_1,b_2) \leq L$ that lie on two different sides of the triangle
\emph{corresponding vertices}.
Note that although the particular corresponding vertices $b_1,b_2$
whose existence we have just shown are found within the sections 
of $\gamma_{\widehat{v_1}}$ and $\gamma_{\widehat{v_2}}$ that join $b$ to $d_1$ and $d_2$, in general, corresponding vertices might be found past either or both of $d_1$ and $d_2$ on those paths.

We claim that it is possible to decide in polynomial time whether a
vertex $b_1$ on $\gamma_{\widehat{v_1}}$ has a corresponding vertex
$b_2$ on $\gamma_{\widehat{v}_2}$ and, if so, find $b_2$ together with
the label in $\Sigma^*$ of a path of length at most $L$ in $\Gamma$
joining $b_1$ to $b_2$.

The justification for this claim is as follows.
Let $l$ be the distance in $\widehat{\Gamma}$ from $b$ to $b_1$.
It is straightforward to define an \slp $\overline{S_B}$ with
value $v_1^{-1}=\val(\cS_B)^{-1}$.
The word $v_1^{-1}$ might not be $\nf$-reduced
but its derived word $\widehat{v_1}^{-1}$ is geodesic,
so we can use Lemma \ref{lem:geo2slextcslp} followed by
Theorem~\ref{thm:convert} to construct an \slp with value
equal to $(\nf(v_1^{-1}))[[:l))$.  Since the number of
words over $\Sigma$ of length at most $L$ is bounded above by the constant
$|\Sigma|^{L+1}$, we can in polynomial time, by Corollary \ref{cor:addtail},
compute \slps with values $\nf(\overline{\cS}_B[[:l)) \eta)$ for all words
$\eta \in \Sigma^*$ of length at most $L$. For each of these, we compute the
length $l'$ of its derived word,
and then check whether its value is equal to $\val(\cS_C[[:l')))$.
If so, then the vertex $b_2$ at distance $l'$ from $b$ in $\Gamma$ corresponds
to $b_1$, and $\eta$ is the required path label from $b_1$ to $b_2$.
Furthermore all such corresponding vertices $b_2$ are found by this procedure.

We now consider two possible situations, as follows.
In Case 1, there is either a vertex $a'$ on $\gamma_{\widehat{v_2}}$
that corresponds to the vertex $a$ of $\gamma_{\widehat{v_1}}$, or
there is a vertex $c'$ on $\gamma_{\widehat{v_1}}$
that corresponds to the vertex $c$ of $\gamma_{\widehat{v_2}}$.
In Case 2, no such vertices $a'$ or $c'$ exist.
By the claim above, we can check which case we are in.

Case 2 is more difficult, so we shall deal with that first and provide
a brief description of the argument for Case 1 at the end.
So suppose that we are in Case 2.
Now the vertex $b$ of $\gamma_{\widehat{v_1}}$ has a corresponding vertex
($b$ itself) on $\gamma_{\widehat{v_2}}$, but the vertex $a$ has no
corresponding vertex on $\gamma_{\widehat{v_2}}$.  We need
to find corresponding vertices $b_1$ and $b_2$ on $\gamma_{\widehat{v_1}}$
and $\gamma_{\widehat{v_2}}$ with the additional property that the vertex
$b'_1$ that is at distance 1 from $b_1$ in $\widehat{\Gamma}$,
on $\gamma_{\widehat{v_1}}$, between $b_1$ and $a$, has no corresponding
vertex on $\gamma_{\widehat{v_2}}$.
We do this, as in the proof of \cite[Theorem 6.7]{HLS}, using the 
technique of {\em binary search} 
to find $b_1$ by testing whether various
vertices that we call $b_t$ have corresponding vertices; 
that is, where $l_0$ is the length of
$\gamma_{\widehat{v_1}}$, we first test if the vertex $b_t$ 
at distance $l_0/2$ from $b$ along $\gamma_{\widehat{v_1}}$ has a
corresponding vertex on $\gamma_{\widehat{v_2}}$. 
If $b_t$ has a corresponding vertex,
and the next vertex
$b_t'$ on $\gamma_{\widehat{v_1}}$ has no corresponding vertex on
$\gamma_{\widehat{v_2}}$, then we set $b_1 := b_t$ and set $b_1' := b_t'$,

%\comment{DFH: new paragraph}
Otherwise, our next choice for $b_t$ is either
the vertex at distance $l_0/4$ from $b$ 
along $\gamma_{\widehat{v_1}}$
(when our first choice  for $b_t$
has no corresponding vertex) or at distance $3l_0/4$ from $b$
(when our first choices for both $b_t$ and $b_t'$ have corresponding vertices).
We continue searching in this way,
each time in one half of the previous interval of $\gamma_{\widehat{v_1}}$,
until we find the vertices $b_1,b_2,b_1'$ that satisfy the required conditions. 
Note that the time taken for this search is logarithmic in the length 
of $\widehat{v_1}$ and so polynomial in the size of its defining \slp $\cS_B$.

If  $b_1'$ were distance greater than $K$ from  $d_1$ along
$\gamma_{\widehat{v_1}}$ within the section joining  $d_1$ to $b$, then
Proposition~\ref{prop:backup} would imply the existence of a corresponding
vertex for
$b_1'$ on $\gamma_{\widehat{v_2}}$; but we know there is no such vertex.
It follows that 
\[d_{\gamma_{\widehat{v_1}}}(a,b_1') \leq  
d_{\gamma_{\widehat{v_1}}}(a,d_1) + K-1, \ \hbox{\rm and}\ 
d_{\gamma_{\widehat{v_1}}}(a,b_1) \leq  
d_{\gamma_{\widehat{v_1}}}(a,d_1) + K.\]

We claim that
$$d_{\gamma_{\widehat{v_2}}}(c,b_2) \le d_{\gamma_{\widehat{v_2}}}(c,d_2) + 3L+\delta.$$
Since $d_{\widehat{\Gamma}}(d_1,d_2) \le \delta$ and
$d_{\widehat{\Gamma}}(b_1,b_2) \le L$, this follows from the triangle inequality
if $d_{\widehat{\Gamma}}(d_1,b_1) \le 2L$. So suppose not. Then, since $L\ge K$,
$b_1$ must be between $a$ and $d_1$ on $\widehat{v_1}$;
note that this position for $b_1$ is {\em not} as suggested in
Figure~\ref{fig:hyptri}. 
But, by Proposition~\ref{prop:backup}, if
$d_{\gamma_{\widehat{v_2}}}(c,b_2) >  d_{\gamma_{\widehat{v_2}}}(c,d_2) + K$, then
there is a point $b_1''$ on $\widehat{v_1}$ 
between $b$ and $d_1$  with
$d_{\Gamma}(b_2,b_1'') \le L$, and hence also
$d_{\widehat{\Gamma}}(b_2,b_1'') \le L$. But then we have
$2L \ge d_{\widehat{v_1}}(b_1,b_1'') \ge d_{\widehat{v_1}}(b_1,d_1) > 2L$, 
a contradiction.
So $d_{\gamma_{\widehat{v_2}}}(c,b_2) \le d_{\gamma_{\widehat{v_2}}}(c,d_2) + K$ and,
since $K < 3L+\delta$, the claim also holds in this case.

So we have
\begin{eqnarray*}
d_{\Gamma}(b_1,b_2) \leq L,\quad
d_{\gamma_{\widehat{v_1}}}(a,b_1) &\leq&
    d_{\gamma_{\widehat{v_1}}}(a,d_1) + K, \\ 
d_{\gamma_{\widehat{v_2}}}(c,b_2) &\leq& d_{\gamma_{\widehat{v_2}}}(c,d_2) + 
3L + \delta.
\end{eqnarray*} 
Now it follows from Proposition~\ref{prop:backup}
that any vertex on $\gamma_{\widehat{v_1}}$ between $a$ and $d_1$
that is at distance at least $K$ along the curve from $d_1$, must
have a corresponding vertex on $\gamma_{\widehat{v_3}}$ that lies between
$a$ and $d_3$.
So now define $a_1$ to be the vertex on $\gamma_{\widehat{v_1}}$
between $a$ and $d_1$ that is distance $2K$ along the curve from $b_1$
(if there is no such vertex, then define $a_1$ to be $a$).
Then there is a vertex $a_3$ on $\gamma_{\widehat{v_3}}$ between $a$ and $d_3$
that corresponds to $a_1$.

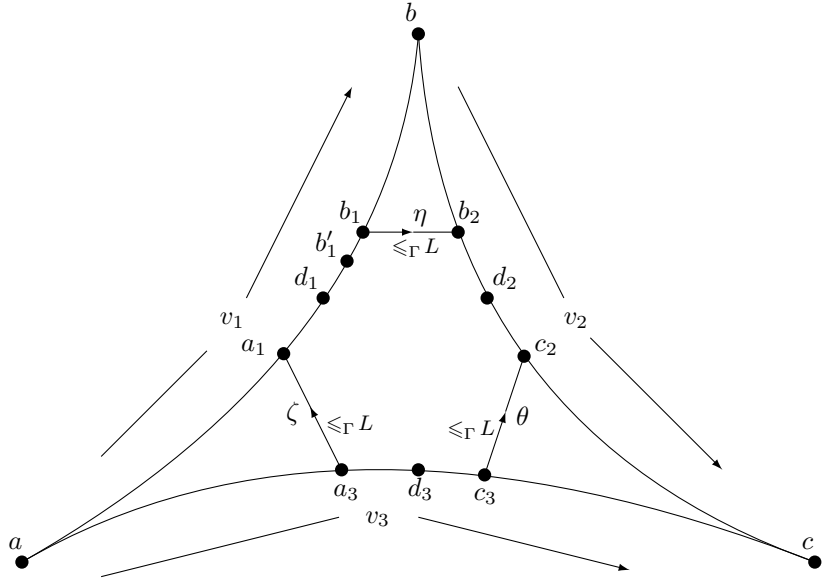
\begin{figure}
\begin{picture}(200,200)(-20,0)
\put(0,0){\circle*{5}} \put(-5,5){$a$}
\put(75,90){$v_1$}
\put(85,100){\vector(1,2){40}}
\put(70,80){\line(-1,-1){40}}
\put(114,100){\circle*{5}} \put(103,105){$d_1$}
\put(176,100){\circle*{5}} \put(178,105){$d_2$}
\put(150,35){\circle*{5}} \put(146,25){$d_3$}

\put(150,200){\circle*{5}} \put(145,205){$b$}
\put(129,125){\circle*{5}} \put(120,130){$b_1$}
\put(123,114){\circle*{5}} \put(111,118){$b'_1$}
\put(165,125){\circle*{5}} \put(165,130){$b_2$}
\put(129,125){\vector(1,0){19}} \put(148,125){\line(1,0){17}}
\put(140,117){\footnotesize{$\leqslant_\Gamma\! L$}}
\put(148,129){$\eta$}

\put(99,79){\circle*{5}} \put(83,80){$a_1$}
\put(121,35){\circle*{5}} \put(118,25){$a_3$}
\put(121,35){\vector(-1,2){12}} \put(110,57){\line(-1,2){10}}
\put(115,50){\footnotesize{$\leqslant_\Gamma\! L$}}
\put(100,54){$\zeta$}

\put(190,78){\circle*{5}} \put(194,80){$c_2$}
\put(175,33){\circle*{5}} \put(171,23.5){$c_3$}
\put(175,33){\vector(1,3){8}} \put(183,57){\line(1,3){7}}
\put(161,48){\footnotesize{$\leqslant_\Gamma\! L$}}
\put(187,52){$\theta$}

\put(205,90){$v_2$}
\put(205,100){\line(-1,2){40}}
\put(215,85){\vector(1,-1){50}}
\put(300,0){\circle*{5}} \put(295,5){$c$}
\put(130,15){$v_3$}
\put(120,17){\line(-4,-1){90}}
\put(150,17){\vector(4,-1){80}}
\qbezier(0,0)(140,80)(150,200)
\qbezier(0,0)(120,70)(300,0)
\qbezier(300,0)(160,50)(150,200)
\end{picture}
\caption{The hyperbolic triangle}\label{fig:hyptri}
\end{figure}

Similarly let $c_2$ be the vertex on $\gamma_{\widehat{v_2}}$
between $c$ and $d_2$ that is distance $K+3L+\delta$ along the curve from $b_2$
(if there is no such vertex, then define $c_2$ to be $c$).
Then there is a vertex $c_3$ on $\gamma_{\widehat{v_3}}$
that corresponds to $c_2$.

There exist words $\zeta$, $\eta$ and $\theta$ over $\Sigma$ each of
length at most $L$, that label paths in $\Gamma$ from $a_3$ to $a_1$, $b_1$ to
$b_2$, and $c_3$ to $c_2$.
We know $\eta$ already, because (by the claim above) we found it
when we defined $b_1$ and $b_2$; to progress further,
we need to find $\zeta$ and $\theta$.
We find these through an exhaustive search process, as we shall now describe.
We generate all possible word pairs $(\zeta,\theta)$ (that is, word pairs of
lengths at most $L$), and check their validity, until we find a solution. 

So suppose that $(\zeta,\theta)$ is a candidate pair.
Define $k_1,l_1,k_2,l_2$ to be the integers for which
$v_1[[k_1:l_1))$ labels the path in $\Gamma$ from 
$a_1$ to $b_1$ and $v_2[[k_2:l_2))$ 
the path from $b_2$ to $c_2$; then $v_1[[k_1:l_1)) = \val(\cS_B[[k_1:l_1)))$
and $v_2[[k_2:l_2)) = \val(\cS_C[[k_2:l_2)))$.
The $\nf$-reduced representative of $v_1[[:k_1))\zeta^{-1}$ is the value of the
\tcslp $\cS_B[[:k_1)) \langle \emptyword,\zeta \rangle$,
and we can find an \slp $\cS_1$
with the same value in polynomial time, by Theorem~\ref{thm:convert}.

The word $\zeta v_1[[k_1:l_1))\eta v_2[[k_2:l_2))\theta^{-1}$
has length at most $3K+6L+\delta$ over $\widehat{\Sigma}$,
and is the value of the \cslp 
$\zeta\cS_B[[k_1:l_1))\eta\cS_C[[k_2:l_2))\theta^{-1}$.
By Proposition \ref{prop:slpshort} applied with $\kappa = 3K+6L+\delta$,
we can (in polynomial time) find an \slp $\cS_2$
with value $\nf(\zeta v_1[[k_1:l_1))\eta v_2[[k_2:l_2))\theta^{-1}).$
The $\nf$-reduced representative of $\theta v_2[[l_2:))$ is the value of the
\tcslp $\cS_C[[l_2:)) \langle \theta,\emptyword \rangle$, and, again,
we can find an \slp $\cS_3$ with the same value in polynomial time.
Now the word $\val(\cS_1\cS_2\cS_3)$ represents the same element of $G$ as
each of the words $v_1v_2$ and $v_3$, and will be equal as a word to $v_3$ if
and only if it is $\nf$-reduced.

By Proposition~\ref{prop:slp_results}\,(v), we can test in polynomial time
whether $\val(\cS_1\cS_2\cS_3)$ is $\nf$-reduced; if it is,
then $\cS_1\cS_2\cS_3$ is an \slp for $v_3$.
We define a \tcslp $\cT_1:= \cT_B[[:k_1)) \langle \emptyword,\zeta \rangle$
as an extension of $\cT_B$ by a single variable $S_{\cT_1}$,
and similarly $\cT_3 := \cT_C[[l_2:)) \langle \theta,\emptyword \rangle$ as an
extension of $\cT_C$. The \tcslps  $\cT_1$ and $\cT_3$ have the same values
as $\cS_1,\cS_3$, respectively, and
the concatenation $\cT_1\cS_2\cT_3$ is a \tcslp with value $v_3$.
We set our copy of $A$ within $\cT$ to be the start variable of that \tcslp,
and define $\rho_\cT(A)$ accordingly.

(The reason that we do not simply adjoin the \slps $\cS_1$ and $\cS_3$ to
$\cT$ is that, if we did that, then we would be unable to prove that the
$|\cT|$ remains bounded throughout the complete process by a polynomial
function of $|\cG|$. This explains why we needed to introduce the concepts
of \tslp and \tcslp.)

We shall now briefly describe the corresponding argument in Case 1, which is
similar to that for Case 2, but simpler. Suppose that there is a vertex $c'$ on
$\gamma_{\widehat{v_1}}$ that corresponds to the vertex $c$ of
$\gamma_{\widehat{v_2}}$ - the other case is similar.
Then, as explained earlier,
we can locate $c'$, we can find $\eta \in \Sigma^*$ that labels a path of
length at most $L$ in $\Gamma$ from $c$ to $c'$, and we can calculate the
integer $k_1$ such that $v_1[[:k_1))$ is the prefix of $v_1$ from $a$ to $c'$.

Let $a_1$ be the vertex on $\gamma_{\widehat{v_1}}$ that is between $a$ and $c'$
and is at distance $K$ in $\gamma_{\widehat{v_1}}$ from $c'$ (or $a_1=a$
if there is no such vertex), and let $k_2$ be its distance from $a$ along
$v_1$.  Then by Proposition~\ref{prop:backup}, $a_1$ has
a corresponding vertex $a_3$ on $\gamma_{\widehat{v_3}}$.
Let $\zeta \in \Sigma^*$ be the label of a path of length at most $L$ in
$\Gamma$ from $a_3$ to $a_1$.  Then
$$v_3 = \nf(v_1[[:k_2))\zeta^{-1})\nf(\zeta v_1[[k_2:k_1))\eta^{-1})$$
and, much as in Case 2, we can find $\zeta$ by exhaustive search,
then define the \tcslp $\cT$ as the concatenation $\cT_1\cS_2$,
where $\val(\cT_1) = \nf(v_1[[:k_2)\zeta^{-1}))$, and
$\val(\cS_2) = \nf(\zeta v_1[[k_2:k_1))\eta^{-1})$.

In both Cases 1 and 2,
we observe that all of the variables of $\cT_B$ and $\cT_C$ were defined during
the processing of other variables, but a copy of $A$, the variables
$S_{\cT_1}$ and $S_{\cT_2}$ and also the variables of $\cS_2$
are added to $V_\cT$ during the processing of the variable $A$ of $\cG$,
and their images under $\rho_T$ are correspondingly defined.
Since the length of $\cS_2$ as a word over $\widehat{\Sigma}$ is bounded by the
constant $3K+6L+\delta$
we know from Proposition~\ref{prop:slpshort} that its size is bounded by a
constant multiple of $\log(|\val(\cS_2)|)$ and hence,
by Proposition~\ref{prop:slp_results}\,(ii), of $|\cG|$.

So after all variables of $\cG$ have been processed, the size of the
final \tcslp $\cT$ is bounded by a quadratic function of $\cG$. This
completes the proof.
\end{proofof}

\begin{proofof}{Lemma \ref{lem:geo2slextcslp}}
As we observed in Remark~\ref{rem:subword},
since all components of $w = \val(\cG)$ already have roots,
it follows from Proposition~\ref{prop:comproot} that
for each variable $A$ of $\cG$, all occurrences of $\val(A)$ as subwords
of $w$ that are derived from $A$ consist either of complete components of $w$
together with subwords in $(\Sigma \setminus \cH)^*$,
or of subwords of a component. So the assumption that $\hat{w}$ is a
geodesic word implies that $\widehat{\val_\cG(A)}$ is also geodesic for all
variables $A$ of $\cG$.

We follow the proof of Theorem \ref{thm:slextcslp}.
When we consider the variable $A$ with $\rho(A) = BC$, we know that
$\widehat{\val_\cG(A)} = \widehat{\val_\cG(B)}\widehat{\val_\cG(C)}$
is geodesic, and so the concatenation of the consecutive sides
$ab$ and $bc$ of the associated hyperbolic triangle in $\widehat{\Gamma}$,
which is labelled by the word $\widehat{uv}=\hat{u}\hat{v}$,
is a geodesic path in $\Gamma$.

Proposition \ref{prop:backup} now implies that
the vertices $a_1$ and $c_2$ on the two paths at distance $K$ from $b$ in
$\widehat{\Gamma}$ have corresponding vertex on $\gamma_{\widehat{v_3}}$.

After defining $a_1$ and $c_2$ in this way, the rest of the proof is the
same as the proof of Theorem \ref{thm:slextcslp}, with
$\eta = \emptyword$. But it is of course important to point out that we have
not used the conclusion of the lemma in its proof!
\end{proofof}

%%%%%%%%%%%%%%%%%%%%%%%%%%%
%%%%%%%%%%%%%%%%%%%%%%%%%%%%

\end{document}